\documentclass[12pt,reqno]{article}
\usepackage{amssymb,amscd,amsmath,amsthm,color}
\newtheorem{thm}{Theorem}[section]

\newtheorem{lemma}[thm]{Lemma}
\newtheorem{cor}[thm]{Corollary}

\newtheorem{remark}[thm]{Remark}
\newtheorem{example}[thm]{Example}

\numberwithin{equation}{section}

\def\bM{\mathbb{M}}
\def\bP{\mathbb{P}}
\def\bN{\mathbb{N}}
\def\bR{\mathbb{R}}
\def\bC{\mathbb{C}}
\def\Tr{\mathrm{Tr}\,}

\def\id{\mathrm{id}}
\def\ffi{\varphi}

\def\eps{\varepsilon}
\def\diag{\mathrm{diag}}
\def\cF{\mathcal{F}}
\def\convex{\mathrm{convex}}
\def\concave{\mathrm{concave}}

\begin{document}
\baselineskip=16pt
\allowdisplaybreaks

\centerline{\LARGE Concavity of certain matrix trace and norm functions.\ II}
\bigskip
\bigskip
\centerline{\Large
Fumio Hiai\footnote{Supported in part by Grant-in-Aid for Scientific Research (C)21540208. \\
\quad\ \ {\it E-mail address:} hiai.fumio@gmail.com}}

\medskip
\begin{center}
$^1$\,Tohoku University (Emeritus), \\
Hakusan 3-8-16-303, Abiko 270-1154, Japan
\end{center}

\medskip
\begin{abstract}
We refine Epstein's method to prove joint concavity/convexity of matrix trace functions
of Lieb type $\Tr f(\Phi(A^p)^{1/2}\Psi(B^q)\Phi(A^p)^{1/2})$ and symmetric (anti-) norm
functions of the form $\|f(\Phi(A^p)\,\sigma\,\Psi(B^q))\|$, where $\Phi$ and $\Psi$ are
positive linear maps, $\sigma$ is an operator mean, and $f(x^\gamma)$ with a certain power
$\gamma$ is an operator monotone function on $(0,\infty)$. Moreover, the variational method of
Carlen, Frank and Lieb is extended to general non-decreasing convex/concave functions on
$(0,\infty)$ so that we prove joint concavity/convexity of more trace functions of Lieb type.


\end{abstract}

\section{Introduction}

In the present paper we consider two-variable matrix functions
\begin{align}
F(A,B)&=f(\Phi(A^p)^{1/2}\Psi(B^q)\Phi(A^p)^{1/2}), \label{F-1.1}\\
F(A,B)&=f(\Phi(A^p)\,\sigma\,\Psi(B^q)), \label{F-1.2}
\end{align}
where $A,B$ are positive definite matrices, $p,q$ are real parameters, $\Phi,\Psi$ are
(strictly) positive linear maps, $\sigma$ is an operator mean, and $f$ is a real function on
$(0,\infty)$. The problem of our concern is joint concavity/convexity of trace and norm
functions of such $F(A,B)$ as above. The problem originated with seminal papers of Lieb
\cite{Li} and Epstein \cite{Ep} in 1973. In \cite{Li}, motivated by a conjecture on
Wigner-Yanase-Dyson skew information, Lieb established the so-called Lieb concavity/convexity
for the matrix trace function $(A,B)\mapsto\Tr X^*A^pXB^q$, that is a special case of
\eqref{F-1.1} when $\Phi=X^*\cdot X$, $\Psi=\id$ and $f(x)=x$. An equivalent reformulation
is Ando's matrix concavity/convexity of $(A,B)\mapsto A^p\otimes B^q$ in \cite{An}. On the
other hand, in \cite{Ep} Epstein developed a complex function method using theory of Pick
functions, called Epstein's method, to prove concavity of the trace function
$A\mapsto\Tr(X^*A^pX)^{1/p}$.

In these years, big progress in the subject matter has been made by several authors. For
instance, in \cite{CL1,CL2} Carlen and Lieb extensively developed concavity/convexity of the
trace functions of the forms $\Tr(X^*A^pX)^s$ and $\Tr(A^p+B^p)^s$ of Minkowski type. Very
recently, in \cite{CFL} they with Frank made the best use of the variational formulas
discovered in \cite{CL2} to obtain concavity/convexity of the trace functions
\begin{equation}\label{F-1.3}
(A,B)\mapsto\Tr(A^{p/2}B^qA^{p/2})^s,
\end{equation}
a special case of the trace functions of \eqref{F-1.1} with $f(x)=x^s$. In our previous papers
\cite{Hi1,Hi3} we refined Epstein's complex function method to prove joint concavity/convexity
results for the trace functions of \eqref{F-1.1} and for the norm/trace functions of
\eqref{F-1.2} in the case $f(x)=x^s$. For additional relevant results see \cite{CFL,CL2,Hi3} and
references therein. Moreover, it is worth noting that our problem on concavity/convexity of
\eqref{F-1.3} also emerges from recent developments of new R\'enyi relative entropies relevant
to quantum information theory. That is closely related to monotonicity of those relative
entropies under quantum channels (i.e., completely positive and trace-preserving maps), as
mentioned in the last part of \cite{CFL} (see also \cite{AD} and references therein).

The present paper is a continuation of \cite{Hi1,Hi3}. In Sections 2 and 3 we further refine
Epstein's method used in \cite{Hi1,Hi3} and prove concavity/convexity theorems for the trace
functions of \eqref{F-1.1} and for the symmetric (anti-) norm functions of \eqref{F-1.2} when
$f(x^\gamma)$ with a certain power $\gamma$ is an operator monotone function on $(0,\infty)$.
In Section 4 we present a general method to passage from concavity/convexity of symmetric
(anti-) norm functions to that of trace functions, and apply it to obtain some general
concavity/convexity result for the trace functions of \eqref{F-1.2}. In Section 5 we extend
the variational method in \cite{CL2,CFL} to general non-decreasing convex/concave functions on
$(0,\infty)$, which enables us to obtain more concavity/convexity theorems for the trace
functions of \eqref{F-1.1}. To do this, we provide, in the appendix, some variational formulas
for such functions on $(0,\infty)$, which might be of independent interest as a theory of
conjugate functions (or the Legendre transform) on $(0,\infty)$.

\section{Trace functions of Lieb type with operator monotone functions}

For each $n\in\bN$ the $n\times n$ complex matrix algebra is denoted by $\bM_n$. We write
$\bM_n^+:=\{A\in\bM_n:A\ge0\}$, the $n\times n$ positive semidefinite matrices, and
$\bP_n:=\{A\in\bM_n:A>0\}$, the $n\times n$ positive definite matrices. The usual trace on
$\bM_n$ is denoted by $\Tr$. A linear map $\Phi:\bM_n\to\bM_l$ is positive if $A\in\bM_n^+$
implies $\Phi(A)\in\bM_l^+$, and it is strictly positive if $A\in\bP_n$ implies
$\Phi(A)\in\bP_l$. A positive linear map $\Phi:\bM_n\to\bM_l$ is strictly positive if and only
if $\Phi(I_n)\in\bP_l$, where $I_n$ (or simply $I$) is the identity of $\bM_n$.

A real function $h$ on $(0,\infty)$ is said to be {\it operator monotone} (resp., {\it operator
monotone decreasing}) if $A\le B$ implies $h(A)\le h(B)$ (resp., $h(A)\ge h(B)$) for
$A,B\in\bP_n$ of any $n\in\bN$. Obviously, $h$ is operator monotone decreasing if and only
if $-h$ is operator monotone.

Let $n,m,l\in\bN$ and $p,q\in\bR$. Assume that $(p,q)\ne(0,0)$; otherwise our problem is trivial.
Let $f$ be a real function on $(0,\infty)$. Throughout the paper, unless otherwise stated, we
assume that $\Phi:\bM_n\to\bM_l$ and $\Psi:\bM_m\to\bM_l$ are strictly positive linear maps. The
aim of this section is to prove the next theorem concerning joint concavity/convexity of the
trace function of Lieb type
\begin{equation}\label{F-2.1}
(A,B)\in\bP_n\times\bP_m\longmapsto\Tr f(\Phi(A^p)^{1/2}\Psi(B^q)\Phi(A^p)^{1/2}).
\end{equation}
The result was announced in the concluding remarks of \cite{Hi3}. Our strategy for the proof
is to improve so-called Epstein's method \cite{Ep} that was also used in our previous papers
\cite{Hi1,Hi3}.

\begin{thm}\label{T-2.1}
Assume that either $0\le p,q\le1$ or $-1\le p,q\le0$ (hence $p+q>0$ or $<0$ from the assumption
$(p,q)\ne(0,0)$). Let $f$ be a real function on $(0,\infty)$. If $f(x^{p+q})$ is operator
monotone (resp., operator monotone decreasing) on $(0,\infty)$, then \eqref{F-2.1} is jointly
concave (resp., jointly convex).
\end{thm}

When $f(x)=x^s$ with $s\in\bR$, we have an important special case
\begin{equation}\label{F-2.4}
(A,B)\in\bP_n\times\bP_m\longmapsto
\Tr\bigl\{\Phi(A^p)^{1/2}\Psi(B^q)\Phi(A^p)^{1/2}\bigr\}^s.
\end{equation}
The most familiar case where $\Phi=\Psi=\id$ is
\begin{equation}\label{F-2.5}
(A,B)\in\bP_n\times\bP_n\longmapsto\Tr(A^{p/2}B^qA^{p/2})^s.
\end{equation}
Theorem \ref{T-2.1} improves \cite[Theorem 2.1]{Hi3} as follows: If either $0\le p,q\le1$ and
$0\le s\le1/(p+q)$, or $-1\le p,q\le0$ and $1/(p+q)\le s\le0$, then \eqref{F-2.4} is jointly
concave. If either $0\le p,q\le1$ and $-1/(p+q)\le s\le0$, or $-1\le p,q\le0$ and
$0\le s\le-1/(p+q)$, then \eqref{F-2.4} is jointly convex. The concavity assertion, together
with \cite[Proposition 5.1\,(2)]{Hi3}, says that \eqref{F-2.5} is jointly concave if and only if
either $0\le p,q\le1$ and $0\le s\le1/(p+q)$, or $-1\le p,q\le0$ and $1/(p+q)\le s\le0$. This
characterization result was recently established in \cite{CFL} as well. On the other hand, the
convexity assertion was extended to a wide variety of $(p,q,s)$ in \cite{CFL} (and also in
Section 5 of this paper).

A corollary of Theorem \ref{T-2.1} is

\begin{cor}\label{C-2.3}
Assume that $\Phi:\bM_n\to\bM_l$ and $\Psi:\bM_m\to\bM_l$ are unital positive linear maps. Let
$0\le\alpha\le1$. If $f$ is operator monotone (resp., operator monotone decreasing) on
$(0,\infty)$, then
$$
(A,B)\in\bP_n\times\bP_m\longmapsto
\Tr f(\exp\{\alpha\Phi(\log A)+(1-\alpha)\Psi(\log B)\})
$$
is jointly concave (resp., jointly convex).
\end{cor}

\begin{proof}
We may assume that $0<\alpha<1$. It is easy to see that for every $A\in\bP_n$ and $B\in\bP_m$,
$$
\lim_{r\searrow0}\Phi(A^{\alpha r})^{1/r}=\exp(\alpha\Phi(\log A)),\qquad
\lim_{r\searrow0}\Psi(B^{(1-\alpha)r})^{1/r}=\exp((1-\alpha)\Psi(\log B)),
$$
from which it is also easy to verify (see the proof of \cite[Lemma 3.3]{HP} for instance) that
\begin{align*}
&\bigl\{\Phi(A^{\alpha r})^{1/2}\Psi(B^{(1-\alpha)r})
\Phi(A^{\alpha r})^{1/2}\bigr\}^{1/r} \\
&\qquad=\bigl\{\bigl(\Phi(A^{\alpha r})^{1/r}\bigr)^{r/2}
\bigl(\Psi(B^{(1-\alpha)r})^{1/r}\bigr)^r
\bigl(\Phi(A^{\alpha r})^{1/r}\bigr)^{r/2}\bigr\}^{1/r} \\
&\qquad\longrightarrow\exp\{\alpha\Phi(\log A)+(1-\alpha)\Psi(\log B)\}
\end{align*}
as $r\searrow0$. Hence the corollary follows by taking the limit of the concavity/convexity
assertions of Theorem \ref{T-2.1} applied to  $p=\alpha r$ and $q=(1-\alpha)r$, since
$f(x^{(p+q)/r})=f(x)$.
\end{proof}

\noindent
{\it Proof of Theorem \ref{T-2.1}.}\enspace
First, the convexity assertion follows by applying the concavity one to $-f$. So what we need to
prove is that if $h$ is an operator monotone function on $(0,\infty)$ and if either
$0\le p,q\le1$ or $-1\le p,q\le0$, then
\begin{equation}\label{F-2.2}
(A,B)\in\bP_n\times\bP_m\longmapsto
\Tr h\bigl(\{\Phi(A^p)^{1/2}\Psi(B^q)\Phi(A^p)^{1/2}\}^{1/(p+q)}\bigr)
\end{equation}
is jointly concave. Recall \cite[Theorem 1.9]{FHR} that an operator
monotone function $h$ on $(0,\infty)$ admits an integral expression
$$
h(x)=h(1)+bx+\int_{[0,\infty)}{(x-1)(1+\lambda)\over x+\lambda}\,d\mu(\lambda),
$$
where $b\ge0$ and $\mu$ is a finite positive measure on $[0,\infty)$. To prove the assertion,
it suffices to show that \eqref{F-2.2} is jointly concave when $h(x)=\mathrm{const.}$, $h(x)=x$,
and $h(x)=x/(x+\lambda)$, $\lambda\ge0$, separately, and \eqref{F-2.2} is jointly convex when
$h(x)=1/(x+\lambda)$, $\lambda\ge0$. When $h(x)=\mathrm{const.}$, the assertion is trivial, and
when $h(x)=x$ it is contained in \cite[Theorem 2.1]{Hi3}. For the case $h(x)=x/(x+\lambda)$,
it is trivial when $\lambda=0$ so that $h(x)=1$, and when $\lambda>0$, by considering
$h(\lambda x)$ it suffices to show the case $h(x)=x/(x+1)=(1+x^{-1})^{-1}$. For convexity of
\eqref{F-2.2} for $h(x)=1/(x+\lambda)$, when $\lambda>0$, by considering
$h(\lambda x)=\lambda^{-1}(1-x/(x+1))$ the assertion is reduced to concavity for $h(x)=x/(x+1)$,
and when $h(x)=1/x$ it is in \cite[Theorem 2.1]{Hi3}. Thus, it suffices to prove that if either
$0\le p,q\le1$ or $-1\le p,q\le0$, and if $A,H\in\bM_n$ and $B,K\in\bM_m$ are such that $A,B>0$
and $H,K$ are Hermitian, then
\begin{equation}\label{F-2.3}
{d^2\over dx^2}\,\Tr\bigl(I_l+\{\Phi((A+xH)^p)^{1/2}\Psi((B+xK)^q)
\Phi((A+xH)^p)^{1/2}\}^{-1/(p+q)}\bigr)^{-1}\le0
\end{equation}
for every sufficiently small $x>0$.

Here it is worth noting that although \cite[Theorem 2.1]{Hi3} has been referred to in the above
discussion, it is in fact unnecessary in our proof of the theorem. Indeed, once \eqref{F-2.3}
is proved, joint concavity of \eqref{F-2.2} for $h(x)=x$ and joint convexity of \eqref{F-2.2}
for $h(x)=1/x$ are obtained by taking the limits $\lambda x/(x+\lambda)\to x$ as
$\lambda\to\infty$ and $\lambda^{-1}(1-x/(x+\lambda))\to1/x$ as $\lambda\searrow0$, which
are the cases we referred to from \cite[Theorem 2.1]{Hi3} in the above.

Now, assume that $0\le p,q\le1$ (and $p+q>0$). Let $A,H,B,K$ be as in \eqref{F-2.3}, and
set $X(z):=zA+H$ and $Y(z)=zB+K$ for $z\in\bC$. As in the proof of \cite[Theorem 2.1]{Hi3},
we see that the function
$$
F(z):=\Phi(X(z)^p)^{1/2}\Psi(Y(z)^q)\Phi(X(z)^p)^{1/2}
$$
is a well-defined analytic function in the upper half-plane $\bC^+$, for which
$$
\sigma(F(z))\subset\{\zeta\in\bC:\zeta=re^{i\theta},\ r>0,\ 0<\theta<\gamma\pi\},
\qquad z\in\bC^+,
$$
where $\gamma:=p+q\in(0,2]$ and $\sigma(F(z))$ is the set of the eigenvalues of $F(z)$.
Therefore, $F(z)^{-1/\gamma}$ is well-defined in $\bC^+$ via analytic functional calculus by
$\zeta^{-1/\gamma}=r^{-1/\gamma}e^{-i\theta/\gamma}$ for $\zeta=re^{i\theta}$ ($r>0$,
$0<\theta<\gamma\pi$) so that $\sigma(F(z)^{-1/\gamma})$ is included in the lower half-plane
$\bC^-$ for all $z\in\bC^+$. Hence the function $(z^{-1}I+F(z)^{-1/\gamma})^{-1}$ is a
well-defined analytic function in $\bC^+$ for which
$\sigma\bigl((z^{-1}I+F(z)^{-1/\gamma})^{-1}\bigr)\subset\bC^+$ for all $z\in\bC^+$, so
$\Tr(z^{-1}+F(z)^{-1/\gamma})^{-1}\in\bC^+$ for all $z\in\bC^+$. Furthermore, one can choose an
$R>0$ such that $xA+H>0$ and $xB+K>0$ for all $x\in(R,\infty)$. Then $F(z)$ in $\bC^+$ is
continuously extended to $\bC^+\cup(R,\infty)$ so that
\begin{align*}
F(x)&=\Phi((xA+H)^p)^{1/2}\Psi((xB+K)^q)\Phi((xA+H)^p)^{1/2} \\
&=x^\gamma\Phi((A+x^{-1}H)^p)^{1/2}\Psi((B+x^{-1}K)^q)\Phi((A+x^{-1}H)^p)^{1/2},
\quad x\in(R,\infty).
\end{align*}
Therefore, for every $x\in(R,\infty)$ one has
\begin{align*}
&(x^{-1}I+F(x)^{-1/\gamma})^{-1} \\
&\quad=x\bigl(I+\{\Phi((A+x^{-1}H)^p)^{1/2}\Psi((B+x^{-1}K)^q)
\Phi((A+x^{-1}H)^p)^{1/2}\}^{-1/\gamma}\bigr)^{-1}.
\end{align*}
Since $\Tr(x^{-1}I+F(x)^{-1/\gamma})^{-1}\in\bR$ for all $x\in(R,\infty)$, by the reflection
principle we obtain a Pick function $\ffi$ on $\bC\setminus(-\infty,R]$ such that
$$
\ffi(x)=\Tr(x^{-1}I+F(x)^{-1/\gamma})^{-1},\qquad x\in(R,\infty).
$$
Thus, for every $x\in(0,R^{-1})$ we have
$$
x\ffi(x^{-1})=\Tr\bigl(I+\{\Phi((A+xH)^p)^{1/2}\Psi((B+xK)^q)
\Phi((A+xH)^p)^{1/2}\}^{-1/\gamma}\bigr)^{-1}.
$$
Now, in the same way (using Epstein's method) as in the proof of \cite[Theorem 2.1]{Hi3}, it
follows that
$$
{d^2\over dx^2}(x\ffi(x^{-1}))\le0,\qquad x\in(0,R^{-1}),
$$
and hence \eqref{F-2.3} follows when $0\le p,q\le1$.

Next, assume that $-1\le p,q\le0$ (and $p+q<0$). Set $\hat\Phi(A):=\Phi(A^{-1})^{-1}$ for
$A\in\bP_n$ and $\hat\Psi(B):=\Psi(B^{-1})^{-1}$ for $B\in\bP_m$. Then we can write
$$
\bigl\{\Phi(A^p)^{1/2}\Psi(B^q)\Phi(A^p)^{1/2}\bigr\}^{1/(p+q)}
=\bigl\{\hat\Phi(A^{-p})^{1/2}\hat\Psi(B^{-q})\hat\Phi(A^{-p})^{1/2}\bigr\}^{-1/(p+q)}.
$$
Although $\hat\Phi$ and $\hat\Psi$ are no longer linear, the above proof of \eqref{F-2.3} can
work with $\hat\Phi$ and $\hat\Psi$ in place of $\Phi$ and $\Psi$ (see the proof of
\cite[Theorem 2.1]{Hi3} for more detail). Hence we have \eqref{F-2.3} for $-1\le p,q\le0$
as well.\qed

\bigskip
It is obvious that if $p,q\ge0$ and $f$ can continuously extend to $[0,\infty)$, then joint
concavity/convexity in Theorem \ref{T-2.1} holds true, by a simple convergence argument, for
general positive (not necessarily strictly positive) linear maps $\Phi,\Psi$ and general
positive semidefinite matrices $A,B$. This remark may be applicable in a similar situation
throughout the paper.

\section{Norm functions involving operator means}

A {\it symmetric anti-norm} $\|\cdot\|_!$ on $\bM_l^+$ is a non-negative continuous functional
such that $\|\lambda A\|_!=\lambda\|A\|_!$, $\|UAU^*\|_!=\|A\|_!$ and
$\|A+B\|_!\ge\|A\|_!+\|B\|_!$ for all $A,B\in\bM_l^+$, all reals $\lambda\ge0$ and all unitaries
$U$ in $\bM_l$. This notion is the superadditive version of usual {\it symmetric norms} (see
\cite{BH1} for details on anti-norms). The typical example is the {\it Ky Fan $k$-anti-norm}
$\|A\|_{\{k\}}:=\sum_{j=1}^k\lambda_{l+1-j}(A)$ for $1\le k\le l$, the anti-norm version of
{\it Ky Fan $k$-norm} $\|A\|_{(k)}:=\sum_{j=1}^k\lambda_j(A)$, where
$\lambda_1(A)\ge\dots\ge\lambda_l(A)$ are the eigenvalues of $A\in\bM_l^+$ in decreasing order
with multiplicities. For every symmetric norm $\|\cdot\|$ on $\bM_l$ and every $\alpha>0$
a symmetric anti-norm on $\bM_l^+$ is defined as
$$
\|A\|_!:=\begin{cases}\|A^{-\alpha}\|^{-1/\alpha} & \text{if $A$ is invertible}, \\
0 & \text{otherwise},\end{cases}
$$
that is called the {\it derived anti-norm} (see \cite[Proposition 4.6]{BH2}).

Throughout this section we assume that $\sigma$ is an {\it operator mean} in the Kubo-Ando sense
\cite{KA}. We consider joint concavity/convexity of the norm functions
\begin{align}
(A,B)\in\bP_n\times\bP_m&\longmapsto\|f(\Phi(A^p)\,\sigma\,\Psi(B^q))\|, \label{F-3.1}\\
(A,B)\in\bP_n\times\bP_m&\longmapsto\|f(\Phi(A^p)\,\sigma\,\Psi(B^q))\|_! \label{F-3.2}
\end{align}
for symmetric and anti-symmetric norms. Our main theorem is

\begin{thm}\label{T-3.1}
Assume that either $0\le p,q\le1$ or $-1\le p,q\le0$, and let $\gamma:=\max\{p,q\}$ if $p,q\ge0$
and $\gamma:=\min\{p,q\}$ if $p,q\le0$. Let $f$ be a non-negative real function on $(0,\infty)$.
If $f(x^\gamma)$ is operator monotone on $(0,\infty)$, then \eqref{F-3.2} is jointly concave
for every symmetric anti-norm $\|\cdot\|_!$ on $\bM_l^+$. If $f(x^\gamma)$ is operator monotone
decreasing on $(0,\infty)$, then \eqref{F-3.1} is jointly convex for every symmetric norm
$\|\cdot\|$ on $\bM_l$. 
\end{thm}

Note that the above theorem contains \cite[Theorem 3.2]{Hi3} as a particular case where
$f(x)=x^s$. Also, the theorem gives an extension of \cite[Corollary 3.6]{Hi3} when $\sigma$ is
the arithmetic mean. The following is the special case where $B=A$, $\Psi=\Phi$ and $q=p$, which
extends \cite[Theorem 4.1]{Hi3}.

\begin{cor}\label{C-3.2}
If $h$ is a non-negative and operator monotone function on $(0,\infty)$ and $0<p\le1$, then the
functions $A\in\bP_n\mapsto\|h(\Phi(A^p)^{1/p})\|_!$ and $\|h(\Phi(A^p)^{-1/p})\|_!$ are concave
for every symmetric anti-norm $\|\cdot\|_!$, and the functions
$A\in\bP_n\mapsto\|h(\Phi(A^p)^{-1/p})\|$ and $\|h(\Phi(A^{-p})^{1/p})\|$ are convex for
every symmetric norm $\|\cdot\|$.
\end{cor}

To prove the theorem, we first give a lemma on joint concavity of the trace function. Note that
the trace-norm is a symmetric norm and an anti-symmetric norm simultaneously, so the lemma is
indeed a particular case of Theorem \ref{T-3.1}. However, in the next section we will show that
Theorem \ref{T-3.1} induces joint concavity/convexity of the trace function for even more
general functions $f$.

\begin{lemma}\label{L-3.3}
If $0\le p,q\le1$, $\gamma:=\max\{p,q\}$ and $h$ is an operator monotone function on
$(0,\infty)$, then
$$
(A,B)\in\bP_n\times\bP_m\longmapsto
\Tr h\bigl(\{\Phi(A^p)\,\sigma\,\Psi(B^q)\}^{1/\gamma}\bigr)
$$
is jointly concave.
\end{lemma}

\begin{proof}
We may assume that $p=q$. Indeed, let $A_1,A_2\in\bP_n$ and $B_1,B_2\in\bP_m$, and assume
that $p>q$. Since
$$
\Psi\biggl(\biggl({B_1+B_2\over2}\biggr)^q\biggr)\ge
\Psi\biggl(\biggl({B_1^{q/p}+B_2^{q/p}\over2}\biggr)^p\biggr),
$$
the joint concavity assertion in the case $p=q$ implies that
\begin{align*}
&\Tr h\biggl(\biggl\{\Phi\biggl(\biggl({A_1+A_2\over2}\biggr)^p\biggr)
\,\sigma\,\Psi\biggl(\biggl({B_1+B_2\over2}\biggr)^q\biggr)\biggr\}^{1/p}\biggr) \\
&\quad\ge\Tr h\biggl(\biggl\{\Phi\biggl({A_1+A_2\over2}\biggr)^p\biggr)\,\sigma\,
\Psi\biggl(\biggl({B_1^{q/p}+B_2^{q/p}\over2}\biggr)^p\biggr)\biggr\}^{1/p}\biggr) \\
&\quad\ge{1\over2}\bigl[\Tr h\bigl(\{\Phi(A_1^p)\,\sigma\,\Psi((B_1^{q/p})^p)\}^{1/p}\bigr)
+\Tr h\bigl(\{\Phi(A_1^p)\,\sigma\,\Psi((B_1^{q/p})^p)\}^{1/p}\bigr)\bigr] \\
&\quad={1\over2}\bigl[\Tr h\bigl(\{\Phi(A_1^p)\,\sigma\,\Psi(B_1^q)\}^{1/p}\bigr)
+\Tr h\bigl(\{\Phi(A_1^p)\,\sigma\,\Psi(B_1^q)\}^{1/p}\bigr)\bigr].
\end{align*}
In the above we have used monotonicity of $\sigma$ and of $\Tr h(\cdot)$. Now, let $A,H\in\bM_n$
and $B,K\in\bM_m$ be such that $A,B>0$ and $H,K$ are Hermitian. For joint concavity of the given
trace function (when $p=q$), as in the proof of Theorem \ref{T-2.1}, we need to prove that
\begin{equation}\label{F-3.3}
{d^2\over dx^2}\,\Tr\bigl(I_l+\{\Phi((A+xH)^p)\,\sigma\,\Psi((B+xK)^p)\}^{-1/p}\bigr)^{-1}
\le0.
\end{equation}
Set $X(z):=zA+H$ and $Y(z):=zB+K$ for $z\in\bC$. As in the proof of \cite[Theorem 4.3]{Hi1},
it is seen that the function
$$
F(z):=\Phi(X(z)^p)\,\sigma\,\Psi(Y(z)^p)
$$
is an analytic functions in $\bC^+$, for which
$$
\sigma(F(z))\subset\bigl\{\zeta\in\bC:\zeta=re^{i\theta},\ r>0,\ 0<\theta<p\pi\bigr\},
\qquad z\in\bC^+.
$$
Therefore, $F(z)^{-1/p}$ can be defined in $\bC^+$ so that $\sigma(F(z)^{-1/p})\subset\bC^-$
for all $z\in\bC^+$. The remaining proof of \eqref{F-3.3} is similar to that of Theorem
\ref{T-2.1}.
\end{proof}

\noindent
{\it Proof of Theorem \ref{T-3.1}.}\enspace
Let $0\le p,q\le1$ and $\gamma:=\max\{p,q\}$. To prove the first assertion, we need to show
that if $h$ is a non-negative and operator monotone function on $(0,\infty)$, then the functions
\begin{align}
(A,B)\in\bP_n\times\bP_m&\longmapsto
\big\|h\bigl(\{\Phi(A^p)\,\sigma\,\Psi(B^q)\}^{1/\gamma}\bigr)\big\|_!, \label{F-3.4}\\
(A,B)\in\bP_n\times\bP_m&\longmapsto
\big\|h\bigl(\{\Phi(A^{-p})\,\sigma\,\Psi(B^{-q})\}^{-1/\gamma}\bigr)\big\|_! \label{F-3.5}
\end{align}
are jointly concave for every symmetric anti-norm $\|\cdot\|_!$ on $\bM_l^+$.

The proof below is similar to that of \cite[Theorem 3.2]{Hi3}. First, note that $h$ can be
extended to $[0,\infty)$ continuously, i.e., $h(0):=\lim_{x\searrow0}h(x)$. For every
$A_1,A_2\in\bP_n$, $B_1,B_2\in\bP_m$ and for every Ky Fan $k$-anti-norm $\|\cdot\|_{\{k\}}$,
$1\le k\le l$, there exists a rank $k$ projection $E$ commuting with
$\Phi(((A_1+A_2)/2)^p)\,\sigma\,\Psi(((B_1+B_2)/2)^q)$ such that
\begin{align*}
&\bigg\|h\biggl(\biggl\{\Phi\biggl(\biggl({A_1+A_2\over2}\biggr)^p\biggr)\,\sigma\,
\Psi\biggl(\biggl({B_1+B_2\over2}\biggr)^q\biggr)\biggr\}^{1/\gamma}\biggr)
\bigg\|_{\{k\}} \\
&\quad=\Tr h\biggl(\biggl\{E\biggl(\Phi\biggl(\biggl({A_1+A_2\over2}\biggr)^p\biggr)
\,\sigma\,\Psi\biggl(\biggl({B_1+B_2\over2}\biggr)^q\biggr)\biggr)E
\biggr\}^{1/\gamma}\biggr)-h(0)\Tr(I_l-E) \\
&\quad=\lim_{\eps\searrow0}\Tr h\biggl(\biggl\{\biggl((E+\eps I_l)
\Phi\biggl(\biggl({A_1+A_2\over2}\biggr)^p\biggr)(E+\eps I_l)\biggr) \\
&\hskip3cm\,\sigma\,
\biggl((E+\eps I_l)\Psi\biggl(\biggl({B_1+B_2\over2}\biggr)^q\biggr)(E+\eps I_l)\biggr)
\biggr\}^{1/\gamma}\biggr)-h(0)\Tr(I_l-E).
\end{align*}
By Lemma \ref{L-3.3} applied to the strictly positive linear maps
$(E+\eps I_l)\Phi(\cdot)(E+\eps I_l)$ and $(E+\eps I_l)\Psi(\cdot)(E+\eps I_l)$ we obtain
\begin{align*}
&\Tr h\biggl(\biggl\{\biggl((E+\eps I_l)
\Phi\biggl(\biggl({A_1+A_2\over2}\biggr)^p\biggr)(E+\eps I_l)\biggr) \\
&\qquad\qquad\qquad\,\sigma\,
\biggl((E+\eps I_l)\Psi\biggl(\biggl({B_1+B_2\over2}\biggr)^q\biggr)(E+\eps I_l)\biggr)
\biggr\}^{1/\gamma}\biggr) \\
&\quad\ge{1\over2}\bigl[\Tr h\bigl(\{(E+\eps I_l)(\Phi(A_1^p)\,\sigma\,\Psi(B_1^q))
(E+\eps I_l)\}^{1/\gamma}\bigr) \\
&\qquad\qquad\qquad+\Tr h\bigl(\{(E+\eps I_l)(\Phi(A_2^p)\,\sigma\,\Psi(B_2^q))
(E+\eps I_l)\}^{1/\gamma}\bigr)\bigr] \\
&\longrightarrow{1\over2}
\bigl[\Tr h\bigl(\{E(\Phi(A_1^p)\,\sigma\,\Psi(B_1^q))E\}^{1/\gamma}\bigr)
+\Tr h\bigl(\{E(\Phi(A_2^p)\,\sigma\,\Psi(B_2^q))E\}^{1/\gamma}\bigr)\bigr]
\end{align*}
as $\eps\searrow0$. Since
$$
\lambda_j^\uparrow(ECE)\ge\lambda_j^\uparrow(C),\qquad C\in\bM_l^+,\ \ j=1,\dots,k,
$$
where $\lambda_j^\uparrow(C)$, $1\le j\le l$, denote the eigenvalues of $C$ in increasing order
with multiplicities. We have
\begin{align*}
&\Tr h\bigl(\{E(\Phi(A_1^p)\,\sigma\,\Psi(B_1^q))E\}^{1/\gamma}\bigr)-h(0)\Tr(I_l-E) \\
&\quad=\sum_{j=1}^kh\bigl(\{\lambda_j^\uparrow(E(\Phi(A_1^p)\,\sigma\,\Psi(B_1^q))E)
\bigr\}^{1/\gamma}\bigr)
\ge\sum_{j=1}^kh\bigl(\{\lambda_j^\uparrow(\Phi(A_1^p)\,\sigma\,\Psi(B_1^q))
\bigr\}^{1/\gamma}\bigr) \\
&\quad=\sum_{j=1}^k\lambda_j^\uparrow\bigl(h\bigl(\{\Phi(A_1^p)\,\sigma\,
\Psi(B_1^q)\}^{1/\gamma}\bigr)\bigr)
=\big\|h\bigl(\{\Phi(A_1^p)\,\sigma\,\Psi(B_1^q)\}^{1/\gamma}\big\|_{\{k\}}
\end{align*}
and similarly
$$
\Tr h\bigl(\{E(\Phi(A_2^p)\,\sigma\,\Psi(B_2^q))E\}^{1/\gamma}\bigr)-h(0)\Tr(I_l-E)
\ge\big\|h\bigl(\{\Phi(A_1^p)\,\sigma\,\Psi(B_1^q)\}^{1/\gamma}\big\|_{\{k\}}.
$$
Combining the above estimates yields
\begin{align*}
&\bigg\|h\biggl(\bigg\{\Phi\biggl(\biggl({A_1+A_2\over2}\biggr)^p\biggr)
\,\sigma\,\Psi\biggl(\biggl({B_1+B_2\over2}\biggr)^q\biggr)\biggr\}^{1/\gamma}
\bigg\|_{\{k\}} \\
&\quad\ge{1\over2}\Bigl[\big\|h\bigl(\{\Phi(A_1^p)\,\sigma\,
\Psi(B_1^q)\}^{1/\gamma}\bigr)\big\|_{\{k\}}
+\big\|h\bigl(\{\Phi(A_1^p)\,\sigma\,\Psi(B_1^q)\}^{1/\gamma}\bigr)\big\|_{\{k\}}\Bigr] \\
&\quad={1\over2}\Big\|h\bigl(\{\Phi(A_1^p)\,\sigma\,
\Psi(B_1^q)\}^{1/\gamma}\bigr)^\uparrow
+h\bigl(\{\Phi(A_2^p)\,\sigma\,
\Psi(B_2^q)\}^{1/\gamma}\bigr)^\uparrow\Big\|_{\{k\}},
\end{align*}
where $C^\uparrow$ for $C\in\bM_l^+$ denotes the diagonal matrix
$\diag(\lambda_1^\uparrow(C),\dots,\lambda_l^\uparrow(C))$. Therefore, by \cite[Lemma 4.2]{BH1}
we have, for any symmetric anti-norm $\|\cdot\|_!$,
\begin{align*}
&\bigg\|h\biggl(\bigg\{\Phi\biggl(\biggl({A_1+A_2\over2}\biggr)^p\biggr)
\,\sigma\,\Psi\biggl(\biggl({B_1+B_2\over2}\biggr)^q\biggr)\biggr\}^{1/\gamma}
\bigg\|_! \\
&\quad\ge{1\over2}\Big\|h\bigl(\{\Phi(A_1^p)\,\sigma\,
\Psi(B_1^q)\}^{1/\gamma}\bigr)^\uparrow
+h\bigl(\{\Phi(A_2^p)\,\sigma\,
\Psi(B_2^q)\}^{1/\gamma}\bigr)^\uparrow\Big\|_! \\
&\quad\ge{1\over2}\Bigl[\big\|h\bigl(\{\Phi(A_1^p)\,\sigma\,
\Psi(B_1^q)\}^{1/\gamma}\bigr)\big\|_!
+\big\|h\bigl(\{\Phi(A_2^p)\,\sigma\,
\Psi(B_2^q)\}^{1/\gamma}\bigr)\big\|_!\Bigr],
\end{align*}
proving joint concavity of \eqref{F-3.4}.

To prove joint concavity of \eqref{F-3.5}, we note that
$$
\{\Phi(A^{-p})\,\sigma\,\Psi(B^{-q})\}^{-1/\gamma}
=\{\hat\Phi(A^p)\,\sigma^*\,\hat\Psi(B^q)\}^{1/\gamma},
$$
where $\hat\Phi(A):=\Phi(A^{-1})^{-1}$ for $A\in\bP_n$ and
$X\,\sigma^*\,Y:=(X^{-1}\,\sigma\,Y^{-1})^{-1}$, the adjoint operator mean. Note that
Lemma \ref{L-3.3} holds true when $\Phi,\Psi$ are replaced with $\hat\Phi,\hat\Psi$,
respectively (with $\sigma^*$ in place of $\sigma$). Hence the above proof for \eqref{F-3.4}
shows the assertion for \eqref{F-3.5} as well.

Next, let $\|\cdot\|$ be a symmetric norm on $\bM_l$. By applying the first assertion to the
operator monotone function $h(x^{-1})^{-1}$ and the derived anti-norm
$\|A\|_!:=\|A^{-1}\|^{-1}$ for $A\in\bP_l$ (and $\|A\|_!=0$ if $A\in\bM_l^+$ is singular), we
see that
\begin{align*}
(A,B)\in\bP_n\times\bP_m&\longmapsto
\big\|h\bigl(\{\Phi(A^p)\,\sigma\,\Psi(B^q)\}^{-1/\gamma}\bigr)\big\|^{-1}, \\
(A,B)\in\bP_n\times\bP_m&\longmapsto
\big\|h\bigl(\{\Phi(A^{-p})\,\sigma\,\Psi(B^{-q})\}^{1/\gamma}\bigr)\big\|^{-1}
\end{align*}
are jointly concave, which implies the second assertion.\qed

\begin{remark}\label{R-3.4}\rm
In the last part of the above theorem, one can take the derived anti-norm
$\|A\|_!:=\|A^{-\alpha}\|^{-1/\alpha}$ with $\alpha>0$, so the second convexity assertion of
Theorem \ref{T-3.1} holds for the function
$\|\{f(\Phi(A^p)\,\sigma\,\Psi(B^q))\}^\alpha\|^{1/\alpha}$ for any $\alpha>0$ more generally
than \eqref{F-3.1}. Note that if $\|\cdot\|$ is a symmetric norm, then
$\|\,|\cdot|^\alpha\|^{1/\alpha}$ is again a symmetric norm for $\alpha\ge1$, but this is not
necessarily so for $0<\alpha<1$.
\end{remark}

\section{Passages from norm functions to trace functions}

In this section we develop an abstract method which provides passages from joint
concavity/convexity of symmetric (anti-) norm functions to that of trace functions in a general
form. The method is then applied to Theorem \ref{T-3.1} (or rather \cite[Theorem 3.2]{Hi3}) so
that we have some general concavity/convexity result for trace functions involving operator
means.

Let $n,m,l$ be fixed and a function $F:\bP_n\times\bP_m\to\bP_l$ be given, for which we consider
the following conditions:
\begin{itemize}
\item[(a)] $(A,B)\in\bP_n\times\bP_m\mapsto\|F(A,B)\|_!$ is jointly concave for every
symmetric anti-norm $\|\cdot\|_!$.
\item[(a)$'$] $(A,B)\in\bP_n\times\bP_m\mapsto\|F(A,B)\|_{\{k\}}$ is jointly concave for the
Ky Fan $k$-anti-norms $\|\cdot\|_{\{k\}}$, $1\le k\le l$.
\item[(b)] $(A,B)\in\bP_n\times\bP_m\mapsto\|F(A,B)^{-1}\|$ is jointly convex for every
symmetric norm $\|\cdot\|$.
\item[(b)$'$] $(A,B)\in\bP_n\times\bP_m\mapsto\|F(A,B)^{-1}\|_{(k)}$ is jointly convex for the
Ky Fan $k$-norms $\|\cdot\|_{(k)}$, $1\le k\le l$.
\item[(c)] $(A,B)\in\bP_n\times\bP_m\mapsto\Tr f(F(A,B))$ is jointly concave for every
non-decreasing concave function $f$ on $(0,\infty)$.
\item[(d)] $(A,B)\in\bP_n\times\bP_m\mapsto\Tr f(F(A,B)^{-1})$ is jointly convex for every
non-decreasing convex function $f$ on $(0,\infty)$.
\end{itemize}

\begin{thm}\label{T-4.1}
Concerning conditions stated above we have
$$
{\rm(a)} \Longleftrightarrow {\rm(a)}' \Longrightarrow {\rm(b)} \Longleftrightarrow {\rm(b)}'
\Longrightarrow {\rm(d)},
$$
$$
{\rm(a)} \Longrightarrow {\rm(c)} \Longrightarrow {\rm(d)}.
$$
\end{thm}

\begin{proof}
(a) $\Rightarrow$ (a)$'$ and (b) $\Rightarrow$ (b)$'$ are trivial. (a)$'$ $\Rightarrow$ (a)
follows from \cite[Lemma 4.2]{BH1} as in the proof of Theorem \ref{T-3.1}, and
(b)$'$ $\Rightarrow$ (b) is similar (see \cite[Proposition 4.4.13]{Hi2}).
(a) $\Rightarrow$ (b) follows from \cite[Propositions 4.6]{BH2}, as used in the last part
of the proof of Theorem \ref{T-3.1}.

(a) $\Rightarrow$ (c).\enspace
Let $A_1,A_2\in\bP_n$ and $B_1,B_2\in\bP_m$. Let $\alpha_1\ge\dots\ge\alpha_l$,
$\alpha_1'\ge\dots\ge\alpha_l'$ and $\alpha_1''\ge\dots\ge\alpha_l''$ be the eigenvalues
of $F((A_1+A_2)/2,(B_1+B_2)/2)$, $F(A_1,B_1)$ and $F(A_2,B_2)$, respectively, in
decreasing order with multiplicities. Joint concavity in (a) for the Ky Fan anti-norms
$\|\cdot\|_{\{k\}}$ means that
$$
\sum_{i=1}^k\alpha_{l+1-i}\ge\sum_{i=1}^k{\alpha_{l+1-i}'+\alpha_{l+1-i}''\over2},
\qquad1\le k\le l,
$$
that is, we have the weak majorization
$$
(-\alpha_{l+1-i})_{i=1}^l\prec_w
\biggl(-{\alpha_{l+1-i}'+\alpha_{l+1-i}''\over2}\biggr)_{i=1}^l.
$$
Now, assume that $f$ is an non-decreasing concave function on $(0,\infty)$. Since $-f(-x)$
is non-decreasing and convex on $(-\infty,0)$, we obtain
$$
-\sum_{i=1}^lf(\alpha_{l+1-i})\le
-\sum_{i=1}^lf\biggl({\alpha_{l+1-i}'+\alpha_{l+1-i}''\over2}\biggr)
$$
and hence
$$
\sum_{i=1}^lf(\alpha_i)\ge\sum_{i=1}^lf\biggl({\alpha_i'+\alpha_i''\over2}\biggr)
\ge\sum_{i=1}^l{f(\alpha_i')+f(\alpha_i'')\over2}
$$
thanks to concavity of $f$. This means that
$$
\Tr f\biggl(F\biggl({A_1+A_2\over2},{B_1+B_2\over2}\biggr)\biggr)\ge
{\Tr f(F(A_1,B_1))+\Tr f(F(A_2,B_2))\over2}.
$$

(b) $\Rightarrow$ (d).\enspace
Let $\alpha_i$, $\alpha_i'$ and $\alpha_i''$ be defined as above corresponding to $F(A,B)^{-1}$
instead of $F(A,B)$. Joint convexity in (b) for the Ky Fan norms $\|\cdot\|_{(k)}$ means the
weak majorization
$$
(\alpha_i)_{i=1}^l\prec_w\biggl({\alpha_i'+\alpha_i''\over2}\biggr)_{i=1}^l.
$$
If $f$ is non-decreasing and convex on $(0,\infty)$, then
$$
\sum_{i=1}^lf(\alpha_i)\le\sum_{i=1}^lf\biggl({\alpha_i'+\alpha_i''\over2}\biggr)
\le\sum_{i=1}^l{f(\alpha_i')+f(\alpha_i'')\over2}
$$
so that
$$
\Tr f\biggl(F\biggl({A_1+A_2\over2},{B_1+B_2\over2}\biggr)^{-1}\biggr)\le
{\Tr f(F(A_1,B_1)^{-1})+\Tr f(F(A_2,B_2)^{-1})\over2}.
$$

(c) $\Rightarrow$ (d) immediately follows from the fact that if $f$ is non-decreasing and
convex on $(0,\infty)$, then $-f(x^{-1})$ is non-decreasing and concave on $(0,\infty)$.
\end{proof}

\begin{cor}\label{C-4.2}
Let $\sigma$ be an operator mean and $f$ be a real function on $(0,\infty)$. Assume that either
$0\le p,q\le1$ or $-1\le p,q\le0$, and let $\gamma:=\max\{p,q\}$ if $p,q\ge0$ and
$\gamma:=\min\{p,q\}$ if $p,q\le0$. If $f(x^\gamma)$ is non-decreasing and concave on
$(0,\infty)$, then
\begin{equation}\label{F-4.1}
(A,B)\in\bP_n\times\bP_m\longmapsto\Tr f(\Phi(A^p)\,\sigma\,\Psi(B^q))
\end{equation}
is jointly concave. If $f(x^{-\gamma})$ is non-decreasing and convex on $(0,\infty)$, then
\eqref{F-4.1} is jointly convex.
\end{cor}

Indeed, Theorem \ref{T-3.1} (also \cite[Theorem 3.2]{Hi3}) implies that the function
$F(A,B):=\{\Phi(A^p)\,\sigma\,\Psi(B^q)\}^{1/\gamma}$ for  $(A,B)\in\bP_n\times\bP_m$ satisfies
condition (a) above, so by Theorem \ref{T-4.1} we have the assertions by rewriting conditions
(c) and (d).

\begin{remark}\rm
When $f$ is non-decreasing and concave on $(0,\infty)$, it is straightforward to see that
the function $\Tr f(\Phi(A^p)\,\sigma\,\Psi(B^q))$ is jointly concave in $(A,B)$ when
$0\le p,q\le1$. Indeed, one has
\begin{align*}
\Phi\biggl(\biggl({A_1+A_2\over2}\biggr)^p\biggr)\,\sigma\,
\Psi\biggl(\biggl({B_1+B_2\over2}\biggr)^p\biggr)
&\ge\biggl({\Phi(A_1^p)+\Phi(A_2^p)\over2}\biggr)\,\sigma\,
\biggl({\Psi(B_1^q)+\Psi(B_2^q)\over2}\biggr) \\
&\ge{\Phi(A_1^p)\,\sigma\,\Psi(B_1^q)+\Phi(A_2^p)\,\sigma\,\Psi(B_2^q)\over2}
\end{align*}
thanks to joint concavity of $\sigma$. Since $\Tr f(\cdot)$ is monotone and concave on $\bP_l$,
we have the conclusion. The real merit of Corollary \ref{C-4.2} is that it holds under the
weaker assumption of $f(x^\gamma)$ being concave.
\end{remark}

\begin{remark}\rm
The assumptions on $p,q$ and $f$ for the joint concavity assertion in Corollary \ref{C-4.2} are
considered optimal from the following facts:
\begin{itemize}
\item Let $p,s\ne0$. If $A\in\bP_2\mapsto\Tr(X^*A^pX)^s$ is concave for any invertible
$X\in\bM_2$, then either $0<p\le1$ and $0<s\le1/p$, or $-1\le p\le0$ and $1/p\le s<0$ (see
\cite[Proposition 5.1\,(1)]{Hi3}).
\item For the case where $p=q=1$ and $\sigma$ is the geometric mean, the numerical function
$f(x^{1/2}y^{1/2})$ must be jointly concave in $x,y>0$, which implies that $f$ is non-decreasing
and concave.
\item Let $p,q\ge0$ and $\gamma:=\max\{p,q\}$. For the case where $\sigma$ is the arithmetic
mean, the numerical function $f(x^p+y^q)$ must be jointly concave in $x,y>0$, which implies
that $f(x^\gamma)$ is concave.
\end{itemize}
\end{remark}

The next corollary gives concavity/convexity of one-variable trace functions of Epstein type.
The first assertion (1) will repeatedly be used in the next section.

\begin{cor}\label{C-4.5}
Let $\Phi:\bM_n\to\bM_l$ be a strictly positive linear map.
\begin{itemize}
\item[\rm(1)] If $0<p\le1$ and $f$ is a non-decreasing concave function on $(0,\infty)$, then
$$
A\in\bP_n\longmapsto\Tr f\bigl(\Phi(A^p)^{1/p}\bigr)\quad
\mbox{and}\quad\Tr f\bigl(\Phi(A^{-p})^{-1/p}\bigr)
$$
are concave.
\item[\rm(2)] Assume that $\Phi$ is CP (i.e., completely positive). If $1\le p\le2$ and $f$ is
a non-decreasing convex function on $(0,\infty)$, then
$$
A\in\bP_n\longmapsto\Tr f\bigl(\Phi(A^p)^{1/p}\bigr)
$$
is convex.
\end{itemize}
\end{cor}

Indeed, (1) is specialization of Corollary \ref{C-4.2} to the case where $B=A$, $\Psi=\Phi$ and
$q=p$. Moreover, it is obvious that Theorem \ref{T-4.1} holds for a one-variable function
$F:\bP_n\to\bP_l$ as well. Applying this to \cite[Theorem 4.2]{Hi3} gives (2).

In particular, Corollary \ref{C-4.5} covers the result in \cite[Theorem 1.1]{CL1} that for
every $X\in\bM_n$ the function $A\in\bM_n^+\mapsto\Tr(X^*A^pX)^{q/p}$ is concave if $0<p\le1$
and $0\le q\le1$, and is convex if $1\le p\le2$ and $q\ge1$.

\begin{remark}\label{R-4.6}\rm
Compared the above (2) with (1) it might be expected that, under the same assumption of (3),
the function $A\in\bP_n\to\Tr f(\Phi(A^{-p})^{-1/p})$ is convex for $1\le p\le2$. In particular,
when $\Phi=K\cdot K^*:\bM_n\to\bM_n$ with an invertible $K\in\bM_n$, this is certainly true
since $\Phi(A^{-p})^{-1/p}=(KA^{-p}K^*)^{-1/p}=(K^{*-1}A^pK^{-1})^{1/p}$. However, it is not
true when $\Phi:\bM_n\to\bM_l$ is a general CP map. For instance, let $E$ be an orthogonal
projection in $\bM_n$, and let $\Phi:\bM_n\to E\bM_nE$ ($\cong\bM_l$ where $l:=\dim E$) be
defined by $\Phi(X)=EXE$ for $X\in\bM_n$. Then the assertion applied to $f(x)=x^s$ for $s\ge1$
would imply that $A\in\bP_n\mapsto\Tr\Phi(A^{-p})^{-s/p}$ is convex for $1\le p\le2$.
For example, let $n=2$, $E=\begin{bmatrix}1/2&1/2\\1/2&1/2\end{bmatrix}$,
$A_1=\begin{bmatrix}1&0\\0&t\end{bmatrix}$ and $A_2=\begin{bmatrix}t&0\\0&1\end{bmatrix}$ for
$t>0$. We then compute
$$
\Tr\Phi\biggl(\biggl({A_1+A_2\over2}\biggr)^{-p}\biggr)^{-s/p}=\biggl({1+t\over2}\biggr)^s,
$$
$$
\Tr\Phi(A_1^{-p})^{-s/p}=\Tr\Phi(A_2^{-p})^{-s/p}=\biggl({1+t^{-p}\over2}\biggr)^{-s/p}.
$$
For any $p,s>0$, since $(1+t)/2>((1+t^{-p})/2)^{-1/p}$ for $t\ne1$, we see that
$A\in\bP_2\mapsto\Tr\Phi(A^{-p})^{-s/p}$ is not convex.
\end{remark}

\section{More general trace functions of Lieb type}

In this section we are concerned with joint concavity/convexity of the functions
\begin{align}
(A,B)\in\bP_n\times\bP_m&\longmapsto
\Tr f(\Phi(A^p)^{1/2}\Psi(B^q)\Phi(A^p)^{1/2}), \label{F-5.1}\\
(A,B)\in\bP_n\times\bP_m&\longmapsto
\Tr f\bigl((\Phi(A^{-p})^{1/2}\Psi(B^{-q})\Phi(A^{-p})^{1/2})^{-1}\bigr). \label{F-5.2}
\end{align}
The form \eqref{F-5.2} is the rewriting of \eqref{F-5.1} by replacing $p,q,f$ with
$-p,-q,f(x^{-1})$.  The form \eqref{F-5.1} of trace functions was already treated in Section 2
but we here consider its joint concavity/convexity problem for more varieties of functions $f$
on $(0,\infty)$ and of real parameters $p,q$. Our strategy here is to extend the method adopted
in \cite[Section 4]{CFL}. To do this, we have to prepare some technical results on variational
formulas of trace functions, which we will summarize in Appendix A.

We first give a lemma which will be useful in the proofs of the theorems below.

\begin{lemma}\label{L-5.1}
Assume that $-1\le q\le0$. Then:
\begin{itemize}
\item[\rm(a)] The function $B\in\bP_m\mapsto\Psi(B^q)^{-1}$ is operator concave, and
$B\in\bP_m\mapsto\Psi(B^{-q})^{-1}$ is operator convex. Hence, if $f$ is an non-decreasing and
concave (resp., convex) function on $(0,\infty)$, then $\Tr f\bigl(\Psi(B^q)^{-1}\bigr)$ (resp.,
$\Tr f\bigl(\Psi(B^{-q})^{-1}\bigr)$ is concave (resp., convex) in $B\in\bP_m$.
\item[\rm(b)] The functions
$$
(X,B)\in\bM_l\times\bP_m\longmapsto X^*\Psi(B^q)X\quad\mbox{and}\quad
X^*\Psi(B^{-q})^{-1}X
$$
are jointly operator convex. Hence, if $f$ is a non-decreasing and convex function on
$(0,\infty)$, then $\Tr f(X^*\Psi(B^q)X)$ and $\Tr X^*\Psi(B^{-q})^{-1}X$ are jointly convex in
$(X,B)\in\bM_l\times\bM_m$.
\end{itemize}
\end{lemma}

\begin{proof}
(a)\enspace
Operator concavity of $B\in\bP_m\mapsto\Psi(B^q)^{-1}$ is \cite[Lemma 3.4]{Hi3}, and operator
convexity of $\Psi(B^{-q})^{-1}$ is similar, so we omit the proof. The latter assertion is
immediately seen from monotonicity and concavity/convexity of $\Tr f(\cdot)$ on $\bP_l$. (Note
that the concavity assertion for $\Tr f\bigl(\Psi(B^q)^{-1}\bigr)$ is also an immediate
consequence of Corollary \ref{C-4.5}\,(1).)

(b)\enspace
First, recall a well-known fact \cite[Theorem 1]{LR} that the function
$(X,Y)\in\bM_l\times\bP_l\mapsto X^*Y^{-1}X$ is jointly operator convex. Let $X_1,X_2\in\bM_n$
and $B_1,B_2\in\bP_m$. Since $B\in\bP_m\mapsto\Psi(B^q)^{-1}$ is operator concave by (a), we
have
$$
\Psi\biggl(\biggl({B_1+B_2\over2}\biggr)^q\biggr)
\le\biggl({\Psi(B_1^q)^{-1}+\Psi(B_2^q)^{-1}\over2}\biggr)^{-1}
$$
and hence
\begin{align*}
&\biggl({X_1+X_2\over2}\biggr)^*\Psi\biggl(\biggl({B_1+B_2\over2}\biggr)^q\biggr)
\biggl({X_1+X_2\over2}\biggr) \\
&\qquad\le\biggl({X_1+X_2\over2}\biggr)^*
\biggl({\Psi(B_1^q)^{-1}+\Psi(B_2^q)^{-1}\over2}\biggr)^{-1}
\biggl({X_1+X_2\over2}\biggr) \\
&\qquad\le{X_1^*\Psi(B_1^q)X_1+X_2^*\Psi(B_2^q)X_2\over2}
\end{align*}
thanks to joint operator convexity mentioned above. For the latter function, since
$((B_1+B_2)/2)^{-q}\ge(B_1^{-q}+B_2^{-q})/2$, we have
$$
\Psi\biggl(\biggl({B_1+B_2\over2}\biggr)^{-q}\biggr)^{-1}
\le\biggl({\Psi(B_1^{-q})+\Psi(B_2^{-q})\over2}\biggr)^{-1}
$$
and thus the assertion follows as above. The latter assertion is immediate as in (a).
\end{proof}

The next theorem gives a sufficient condition for \eqref{F-5.1} and \eqref{F-5.2} to be
jointly concave.

\begin{thm}\label{T-5.2}
Let $f$ be a non-decreasing (resp., non-increasing) function on $(0,\infty)$ and $0\le p,q\le1$.
If either $f(x^{1+p})$ or $f(x^{1+q})$ is concave (resp. convex) on $(0,\infty)$, then the
functions \eqref{F-5.1} and \eqref{F-5.2} are jointly concave (resp., jointly convex).
\end{thm}

\begin{proof}
The convexity assertion follows by applying the concavity one to $-f$. So we may confine the
proof to the concavity assertion. When $p=0$ and $0\le q\le1$, the assertion reduces to
concavity of $B\in\bP_m\mapsto\Tr f(\Phi(I)^{1/2}\Psi(B^q)\Phi(I)^{1/2})$ and
$\Tr f\bigl((\Phi(I)^{1/2}\Psi(B^{-q})\Phi(I)^{1/2})^{-1}\bigr)$. This immediately follows from
operator concavity of $x^q$ (for the former) and from Lemma \ref{L-5.1}\,(a) (for the latter).
The situation is similar when $0\le p\le1$ and $q=0$. So we assume that $0<p,q\le1$ and
$f(x^{1+p})$ is concave on $(0,\infty)$. For every $A\in\bP_n$ and $B\in\bP_m$, by (b) and (c)
of Lemma \ref{L-A.2} with $r=p$ we have
\begin{align*}
\Tr f(\Phi(A^p)^{1/2}\Psi(B^q)\Phi(A^p)^{1/2})
&=\inf_{Y\in\bP_l}\bigl\{\Tr Y\Phi(A^p)^{1/2}\Psi(B^q)\Phi(A^p)^{1/2}
-\Tr\check f(Y)\bigr\} \\
&=\inf_{Y\in\bP_l}\bigl\{\Tr\Phi(A^p)^{1/2}Y\Phi(A^p)^{1/2}\Psi(B^q)
-\Tr\check f(Y)\bigr\}.
\end{align*}
Let $X:=(\Phi(A^p)^{1/2}Y\Phi(A^p)^{1/2})^{1/2}$ and so $Y=\Phi(A^p)^{-1/2}X^2\Phi(A^p)^{-1/2}$;
thus $X$ runs over all $\bP_l$ as $Y$ does. Therefore,
\begin{align}
\Tr f(\Phi(A^p)^{1/2}\Psi(B^q)\Phi(A^p)^{1/2})
&=\inf_{X\in\bP_l}\bigl\{\Tr X^2\Psi(B^q)
-\Tr\check f\bigl(\Phi(A^p)^{-1/2}X^2\Phi(A^p)^{-1/2}\bigr)\bigr\} \nonumber\\
&=\inf_{X\in\bP_l}\bigl\{\Tr X\Psi(B^q)X
-\Tr\check f(X\Phi(A^p)^{-1}X)\bigr\}. \label{F-5.3}
\end{align}
Furthermore, we write
$$
\Tr\check f(X\Phi(A^p)^{-1}X)
=\Tr\check f\bigl(\bigl((X^{-1}\Phi(A^p)X^{-1})^{1/p}\bigr)^{-p}\bigr).
$$
For any fixed $X\in\bP_l$, since $-\check f(x^{-p})$ is non-decreasing and concave on
$(0,\infty)$ by Lemma \ref{L-A.2}\,(c) with $r=p$, it follows from Corollary \ref{C-4.5}\,(1)
that
$$
A\in\bP_n\longmapsto-\Tr\check f(X\Phi(A^p)^{-1}X)
$$
is concave. Since $B\in\bP_m\mapsto\Tr X\Psi(B^q)X$ is concave, we have joint concavity of
\eqref{F-5.1}.

For the function \eqref{F-5.2} we may replace \eqref{F-5.3} with
\begin{align*}
\Tr f\bigl((\Phi(A^{-p})^{1/2}\Psi(B^{-q})\Phi(A^{-p})^{1/2})^{-1}\bigr)
&=\Tr f\bigl(\Phi(A^{-p})^{-1/2}\Psi(B^{-q})^{-1}\Phi(A^{-p})^{-1/2}\bigr) \\
&=\inf_{X\in\bP_l}\bigl\{\Tr X\Psi(B^{-q})^{-1}X
-\Tr\check f(X\Phi(A^{-p})X)\bigr\}.
\end{align*}
For any fixed $X\in\bP_l$, from Corollary \ref{C-4.5}\,(1),
$$
B\in\bP_m\longmapsto\Tr X\Psi(B^{-q})^{-1}X
=\Tr\bigl((X^{-1}\Psi(B^{-q})X^{-1})^{-1/q}\bigr)^q
$$
and
$$
A\in\bP_n\longmapsto-\Tr\check f(X\Phi(A^{-p})X)
=-\Tr\check f\bigl(\bigl((X\Phi(A^{-p})X)^{-1/p}\bigr)^{-p}\bigr)
$$
are concave so that \eqref{F-5.2} is jointly concave.
\end{proof}

For the power functions $f(x)=x^s$ the range of $(p,q,s)$ for joint concavity of \eqref{F-2.4}
covered by Theorem \ref{T-5.2} is the following: $0\le p,q\le1$ and
$0\le s\le\max\{1/(1+p),1/(1+q)\}$, or $-1\le p,q\le0$ and $-\max\{1/(1-p),1/(1-q)\}\le s\le0$,
which is smaller than the best possible range covered by Theorem \ref{T-2.1} (see the paragraph
containing \eqref{F-2.4}). However, Theorem \ref{T-5.2} gains an advantage that it is
applicable to a wider class of functions $f$, as demonstrated in Example \ref{E-A.4}. On the
other hand, the range of $(p,q,s)$ for joint convexity of \eqref{F-2.4} covered by Theorem
\ref{T-5.2} is: $0\le p,q\le1$ and $s\le0$, or $-1\le p,q\le0$ and $s\ge0$, which includes the
range by Theorem \ref{T-2.1}.

The rest of the section is devoted to more results on joint convexity of \eqref{F-5.1} and
\eqref{F-5.2}.

\begin{thm}\label{T-5.3}
Let $f$ be a non-decreasing function on $(0,\infty)$ and $-1<p\le0$. Assume that $f(x^{1+p})$
is convex on $(0,\infty)$. Then:
\begin{itemize}
\item[\rm(1)] For every $q\in[-1,0]\cup[1,2]$ the function \eqref{F-5.1} is jointly convex.
\item[\rm(2)] For every $q\in[-1,0]$ the function \eqref{F-5.2} is jointly convex.
\item[\rm(3)] If $\Psi=\id$ with $\bM_m=\bM_l$, then \eqref{F-5.2} is jointly convex for every
$q\in[1,2]$.
\end{itemize}
\end{thm}

\begin{proof}
Let $-1<p\le0$ and $f$ be a non-constant and non-decreasing function on $(0,\infty)$. Assume
that $f(x^{1+p})$ is convex on $(0,\infty)$, hence so is $f$ .

(1)\enspace
When $p=0$, \eqref{F-5.1} reduces to $B\mapsto\Tr f(\Phi(I)^{1/2}\Psi(B^q)\Phi(I)^{1/2})$, whose
concavity is immediately seen. So assume that $-1<p<0$. For every $A\in\bP_n$ and $B\in\bP_m$,
by (b) and (c) of Lemma \ref{L-A.1} with $r=-p$ we have, as in the proof of Theorem \ref{T-5.2},
\begin{align}
\Tr f(\Phi(A^p)^{1/2}\Psi(B^q)\Phi(A^p)^{1/2})
&=\sup_{Y\in\bP_l}\bigl\{\Tr \Phi(A^p)^{1/2}Y\Phi(A^p)^{1/2}\Psi(B^q)
-\Tr\hat f(Y)\bigr\} \nonumber\\
&=\sup_{X\in\bP_l}\bigl\{\Tr X\Psi(B^q)X
-\Tr\hat f(X\Phi(A^p)^{-1}X)\bigr\}. \label{F-5.4}
\end{align}
For any fixed $X\in\bP_n$, since $\hat f(x^{-p})$ is non-decreasing and concave on $(0,\infty)$
by Lemma \ref{L-A.1}\,(c) with $r=-p$, it follows from Corollary \ref{C-4.5}\,(1) that
$$
A\in\bP_n\longmapsto\Tr\hat f(X\Phi(A^p)^{-1}X)
=\Tr\hat f\bigl(\bigl((X^{-1}\Phi(A^p)X^{-1})^{1/p}\bigr)^{-p}\bigr)
$$
is concave. Moreover, when $q\in[-1,0]\cup[1,2]$, the function $B\in\bP_m\mapsto\Tr X\Psi(B^q)X$
is convex so that joint convexity of \eqref{F-5.1} follows.

(2)\enspace
When $p=0$ or $q=0$, the assertion is immediate from Lemma \ref{L-5.1}\,(a). When $-1<p<0$ and
$-1\le q<0$, we may replace \eqref{F-5.4} with
\begin{align*}
\Tr f\bigl((\Phi(A^{-p})^{1/2}\Psi(B^{-q})\Phi(A^{-p})^{1/2})^{-1}\bigr)
&=\Tr f\bigl(\Phi(A^{-p})^{-1/2}\Psi(B^{-q})^{-1}\Phi(A^{-p})^{-1/2}\bigr) \\
&=\sup_{X\in\bP_l}\bigl\{\Tr X\Psi(B^{-q})^{-1}X
-\Tr\hat f(X\Phi(A^{-p})X)\bigr\}.
\end{align*}
For any fixed $X\in\bP_l$, it follows from Corollary \ref{C-4.5}\,(1) that
$$
A\in\bP_n\longmapsto\Tr\hat f(X\Phi(A^{-p})X)
=\Tr\hat f\bigl(\bigl((X\Phi(A^{-p})X)^{-1/p}\bigr)^{-p}\bigr)
$$
and
$$
B\in\bP_m\longmapsto-\Tr X\Psi(B^{-q})^{-1}X
=-\Tr\bigl((X^{-1}\Psi(B^{-q})X^{-1})^{-1/q}\bigr)^q
$$
are concave. Hence joint convexity of \eqref{F-5.2} follows.

(3)\enspace
When $\Psi=\id$ and $1\le q\le2$, the assertion follows similarly to the above proof of (2)
since $\Tr X\Psi(B^{-q})^{-1}X=\Tr XB^qX$ is convex in $B$.
\end{proof}

\begin{thm}\label{T-5.4}
Let $f$ be a non-decreasing function on $(0,\infty)$ and $1<p\le2$. Assume that $f(x^{p-1})$
is convex on $(0,\infty)$.
\begin{itemize}
\item[\rm(1)] If $\Phi$ is CP, then \eqref{F-5.1} is jointly convex for every $q\in[-1,0]$.
\item[\rm(2)] If $\Phi=\id$ with $\bM_n=\bM_l$, then \eqref{F-5.2} is jointly convex for every
$q\in[-1,0]$.
\end{itemize}
\end{thm}

\begin{proof}
Let $1<p\le2$, $-1\le q\le0$ and $f$ be a non-constant and non-decreasing function on
$(0,\infty)$. Assume that $f(x^{p-1})$ is convex on $(0,\infty)$, hence so is $f$.

(1)\enspace
Assume that $\Phi$ is CP. We take the Stinespring representation
$$
\Phi(Z)=K\pi(Z)K^*,\qquad Z\in\bM_n,
$$
where $\pi:\bM_n\to\bM_{nk}$ is a representation and $K:\bC^{nk}\to\bC^l$ is a linear map
(see, e.g., \cite[Theorem 3.1.2]{Bh2}). We have
\begin{align}
\Tr f(\Phi(A^p)^{1/2}\Psi(B^q)\Phi(A^p)^{1/2})
&=\Tr f(\Psi(B^q)^{1/2}K\pi(A^p)K^*\Psi(B^q)^{1/2}) \nonumber\\
&=\Tr f(\pi(A)^{p/2}K^*\Psi(B^q)K\pi(A)^{p/2})-\alpha_0, \label{F-5.5}
\end{align}
where $\alpha_0:=f(0+)\Tr(I_{nk}-P_0)$ with $P_0$ the orthogonal projection onto the range of
$K^*K$. Assume that $1<p<2$. Letting $\widetilde\Psi(\cdot):=K^*\Psi(\cdot)K:\bM_m\to\bM_{nk}$,
by (b) and (c) of Lemma \ref{L-A.1} with $r=2-p$ we further have
\begin{align*}
&\Tr f(\Phi(A^p)^{1/2}\Psi(B^q)\Phi(A^p)^{1/2}) \\
&\qquad=\sup_{Y\in\bP_{nk}}\bigl\{\Tr Y\pi(A)^{p/2}
\widetilde\Psi(B^q)\pi(A)^{p/2}-\Tr\hat f(Y)\bigr\}-\alpha_0 \\
&\qquad=\sup_{Y\in\bP_{nk}}\bigl\{\Tr \pi(A)^{{p\over2}-1}Y
\pi(A)^{{p\over2}-1}\pi(A)\widetilde\Psi(B^q)\pi(A)-\Tr\hat f(Y)\bigr\}-\alpha_0 \\
&\qquad=\sup_{X\in\bP_{nk}}\bigl\{\Tr X\pi(A)\widetilde\Psi(B^q)\pi(A)X
-\Tr\hat f(X\pi(A)^{2-p}X)\bigr\}-\alpha_0.
\end{align*}
For any fixed $X\in\bP_{nk}$, since $\hat f(x^{2-p})$ is non-decreasing
and concave on $(0,\infty)$ by Lemma \ref{L-A.1}\,(c) with $r=2-p$, Corollary \ref{C-4.5}\,(1)
implies that
$$
A\in\bP_n\longmapsto\Tr\hat f(X\pi(A)^{2-p}X)
$$
is concave. Moreover, $(A,B)\in\bP_n\times\bP_m\mapsto\Tr X\pi(A)\widetilde\Psi(B^q)\pi(A)X$
is convex due to Lemma \ref{L-5.1}\,(b). Here, although $\widetilde\Psi$ is not necessarily
strictly positive, we can simply take the convergence from strictly positive maps. Joint
convexity of \eqref{F-5.1} thus follows. The case $p=2$ also follows by using Lemma
\ref{L-5.1}\,(b) to \eqref{F-5.5} directly.

(2)\enspace
Assume that $\Phi=\id$ with $\bM_n=\bM_l$. The function \eqref{F-5.2} in this case is
$\Tr f(A^{p/2}\Psi(B^{-q})^{-1}A^{p/2})$. As in the above proof for \eqref{F-5.1} we have
$$
\Tr f(A^{p/2}\Psi(B^{-q})^{-1}A^{p/2})
=\sup_{X\in\bP_l}\bigl\{\Tr XA\Psi(B^{-q})^{-1}AX-\Tr\hat f(XA^{2-p}X)\bigr\}.
$$
Hence the assertion follows since $(A,B)\in\bP_n\times\bP_m\mapsto\Tr XA\Psi(B^{-q})^{-1}AX$ is
jointly convex by Lemma \ref{L-5.1}\,(b).
\end{proof}

\begin{remark}\label{R-5.5}\rm
For $1\le q\le2$ (resp., $-1\le q\le0$) joint convexity of \eqref{F-5.2} holds in (3) of
Theorem \ref{T-5.3} (resp., (2) of Theorem \ref{T-5.4}) in a slightly more
general case where $\Psi$ (resp., $\Phi$) $=K\cdot K^*:\bM_l\to\bM_l$ with an invertible
$K\in\bM_l$. However, this is not true when $\Psi:\bM_m\to\bM_l$ (resp., $\Phi:\bM_n\to\bM_l$)
is a general CP map. For instance, let $-1<p\le0$, $f(x)=x^s$ with $s\ge1/(1+p)$, and $E$ be an
orthogonal projection in $\bM_n$. Let $\Psi:\bM_n\to E\bM_nE$ ($\cong\bM_l$) be as defined in
Remark \ref{R-4.6}. Then joint convexity of \eqref{F-5.2} would imply in particular that
$B\in\bP_n\mapsto\Tr\Psi(B^{-q})^{-s}$ for $1\le q\le2$ is convex. But this is not true for
any $q,s>0$ as shown in Remark \ref{R-4.6}. Therefore, (3) of Theorem \ref{T-5.3} is not true
for a general CP map $\Psi$. When $1<p\le2$ and $f(x)=x^s$ with $s\ge1/(p-1)$, the same argument
with $p,\Phi$ in place of $q,\Psi$ works for (2) of Theorem \ref{T-5.4}.
\end{remark}

We finally give the next theorem in the special case where $q=2$, whose proof is essentially
same as that of \cite[Theorem 4.2]{CFL}.

\begin{thm}\label{T-5.6}
Let $-1\le p\le0$ and $f$ be a non-decreasing concave function on $(0,\infty)$ such that
$\lim_{x\to\infty}f(x)/x=0$. Assume that $f(x^{2+p})$ is convex on $(0,\infty)$. If $q=2$ and
$\Phi,\Psi$ are CP, then \eqref{F-5.1} is jointly convex.
\end{thm}

\begin{proof}
Assume that $q=2$ and $\Phi,\Psi$ are CP. Write $\Psi(Z)=K\pi(Z)K^*$ with a representation
$\pi:\bM_m\to\bM_{mk}$ and a linear map $K:\bC^{mk}\to\bC^l$, and let
$\widetilde\Phi(\cdot):=K^*\Phi(\cdot)K:\bM_n\to\bM_{nk}$. Then \eqref{F-5.1} with $q=2$ is
written as $\Tr f(\pi(B)\widetilde\Phi(A^p)\pi(B))-\alpha_0$, where $\alpha_0$ is as given in
\eqref{F-5.5}. When $p=-1$, $\Tr f(\pi(B)\widetilde\Phi(A^{-1})\pi(B))$ is joint convex in
$(A,B)$ by Lemma \ref{L-5.1}\,(b). So we may assume that $-1<p\le0$. For every $A\in\bP_n$ and
$B\in\bP_m$, by Lemma \ref{L-A.2}\,(b) we have
\begin{align*}
&\Tr f(\Phi(A^p)^{1/2}\Psi(B^2)\Phi(A^p)^{1/2}) \\
&\qquad=\inf_{Y\in\bP_{mk}}\bigl\{\Tr\pi(B)Y^{-1-p}\pi(B)\widetilde\Phi(A^p)
-\Tr\check f(Y^{-1-p})\bigr\}-\alpha_0.
\end{align*}
Since $\check f(x^{-1-p})$ is concave on $(0,\infty)$ by Lemma \ref{L-A.2}\,(d) with $r=1+p$,
it follows that $Y\in\bP_l\mapsto\Tr\check f(Y^{-1-p})$ is concave. Hence, by
\cite[Lemma 2.3]{CL2} it suffices to show that
$(A,B,Y)\in\bP_n\times\bP_m\times\bP_{mk}\mapsto\Tr\pi(B)Y^{-1-p}\pi(B)\widetilde\Phi(A^p)$ is
jointly convex. When $p=0$, this holds by \cite[Theorem 1]{LR}. So assume that $-1<p<0$. Since
$\widetilde\Phi$ is CP, we may write $\widetilde\Phi(Z)=\widetilde K\tilde\pi(Z)\widetilde K^*$
with a representation $\tilde\pi:\bM_n\to\bM_{n\tilde k}$ and a linear map
$\widetilde K:\bC^{n\tilde k}\to\bC^{nk}$. Then
$$
\Tr\pi(B)Y^{-1-p}\pi(B)\widetilde\Phi(A^p)
=\Tr\widetilde K^*\pi(B)Y^{-1-p}\pi(B)\widetilde K\tilde\pi(A)^p,
$$
which is jointly convex in $(A,B,Y)$ by \cite[Corollary 2.1]{Li}.
\end{proof}

The theorems proved above of course holds also when the roles of $p,\Phi$ and $q,\Psi$ are
interchanged. In the case of power functions $f(x)=x^s$ we have a variety of ranges of $(p,q,s)$
for joint convexity of \eqref{F-2.4} from the above theorems, which are listed in the following
as well as their counterparts where $p,\Phi$ and $q,\Psi$ are interchanged:
\begin{itemize}
\item[\rm(i)] $0\le p,q\le1$, $s\le0$, or $-1\le p,q\le0$, $s\ge0$, by Theorem \ref{T-5.2}.
\item[\rm(ii)] $-1\le p\le0$, $1\le q\le2$, $s\ge\min\{1/(p+1),1/(q-1)\}$ (with convention
$1/(-1+1)=1/(1-1)=\infty$), and $\Psi$ is CP, by Theorems \ref{T-5.3}\,(1) and \ref{T-5.4}\,(1).
\item[\rm(iii)] $0\le p\le1$, $-2\le q\le-1$, $s\le\max\{1/(p-1),1/(q+1)\}$ (with convention
$1/(1-1)=1/(-1+1)=-\infty$), and $\Psi=\id$,
by Theorems \ref{T-5.3}\,(3) and \ref{T-5.4}\,(2).
\item[\rm(iv)] $-1\le p\le0$, $q=2$, $s\ge1/(2+p)$, and $\Phi,\Psi$ are CP,
by Theorems \ref{T-5.6} and \ref{T-5.4}\,(1).
\end{itemize}

For the function \eqref{F-2.5} (when $\Phi=\Psi=\id$), the convexity results in the cases
(ii), (iii) and (iv) are contained in \cite{CFL}, as seen from
$\Tr(A^{p/2}B^qA^{p/2})^s=\Tr(A^{-p/2}B^{-q}A^{-p/2})^{-s}$. Compared with the necessary
conditions in \cite[Proposition 5.4\,(2)]{Hi3}, the missing region for joint convexity in this
situation is only
$$
-1<p<0,\quad 1<q<2,\quad{1\over p+q}\le s\ (\ne1)
<\min\biggl\{{1\over p+1},{1\over q-1}\biggr\},
$$
and its counterparts where $(p,q,s)$ are replaced with $(-p,-q,-s)$ and/or $p,q$ are
interchanged. Here, note that joint convexity is known when $-1\le p\le0$, $1\le q\le2$,
$s=1\ge1/(p+q)$, due to Ando \cite{An}. In connection with the above missing region, it might
be expected that Theorem \ref{T-5.6} and its proof are also valid in the case where
$-1\le p\le0$, $1\le q\le2$ and $p+q\ge1$. But this does not seem possible due to
\cite[Theorem 3.2]{CFL}.

\appendix

\section{Variational formulas of trace functions}

In this appendix we provide some variational formulas, which have played an essential role in
Section 5, but which may also be of independent interest. For the convenience in exposition let
us introduce the following classes of functions on $(0,\infty)$:
\begin{itemize}
\item $\cF_\convex^\nearrow(0,\infty)$ is the set of non-decreasing convex real functions $f$
on $(0,\infty)$ such that $\lim_{x\to\infty}f(x)/x=+\infty$.
\item $\cF_\concave^\nearrow(0,\infty)$ is the set of non-decreasing concave real functions $f$
on $(0,\infty)$ such that $\lim_{x\to\infty}f(x)/x=0$.
\end{itemize}
Note that affine functions $ax+b$ ($a\ge0$) are excluded from $\cF_\convex^\nearrow(0,\infty)$,
and so are $ax+b$ ($a>0$) from $\cF_\concave^\nearrow(0,\infty)$. 

\begin{lemma}\label{L-A.1}
\begin{itemize}
\item[\rm(a)] For each $f\in\cF_\convex^\nearrow(0,\infty)$ define
$$
\hat f(t):=\sup_{x>0}\{xt-f(x)\},\qquad t\in(0,\infty).
$$
Then $\hat f\in\cF_\convex^\nearrow(0,\infty)$ and $f\mapsto\hat f$ is an involutive bijection
on $\cF_\convex^\nearrow(0,\infty)$, i.e., $\hat{\hat f}=f$ for all
$f\in\cF_\convex^\nearrow(0,\infty)$.
\item[\rm(b)] For every $f\in\cF_\convex^\nearrow(0,\infty)$ and $B\in\bM_n^+$,
$$
\Tr f(B)=\sup_{A\in\bP_n}\bigl\{\Tr AB-\Tr\hat f(A)\bigr\},
$$
where $f$ is continuously extended to $[0,\infty)$.
\item[\rm(c)] Let $f$ be a non-constant and non-decreasing function on $(0,\infty)$ and $0<r<1$.
Then $f(x^{1-r})$ is convex on $(0,\infty)$ if and only if $f\in\cF_\convex^\nearrow(0,\infty)$
and $\hat f(x^r)$ is concave on $(0,\infty)$.
\end{itemize}
\end{lemma}

\begin{proof}
(a)\enspace
Let $f\in\cF_\convex^\nearrow(0,\infty)$ and $t\in(0,\infty)$. Since
$f(0+):=\lim_{x\to0+}f(x)$ exists in $\bR$ and
$xt-f(x)=x(t-f(x)/x)\to-\infty$ as $x\to\infty$, it follows that $\hat f(t)$ is defined as a
finite value. By definition it is clear that $\hat f$ is convex and non-decreasing. For any
$x>0$ fixed, since $\hat f(t)/t\ge x-f(x)/t\to x$ as $t\to\infty$, we have
$\lim_{t\to\infty}\hat f(t)/t=+\infty$, so $\hat f\in\cF_\convex^\nearrow(0,\infty)$. To
show that $f\mapsto\hat f$ is an involutive bijection, we appeal to the duality of conjugate
functions (or the Legendre transform) on $\bR$. For each $f\in\cF_\convex^\nearrow(0,\infty)$
we extend $f$ to a continuous convex function $\bar f$ on the whole $\bR$ by $\bar f(x):=f(0+)$
for $x\le0$. Then it is plain to see that the conjugate function
$\bar f^*(t):=\sup_{x\in\bR}\{xt-\bar f(x)\}$ is
$$
\bar f^*(t)=\begin{cases}+\infty & \text{if $t<0$}, \\
-f(0+)=\hat f(0+) & \text{if $t=0$}, \\
\hat f(t) & \text{if $t>0$}.
\end{cases}
$$
Due to the duality for conjugate functions, we have for $x>0$,
$$
f(x)=\sup_{t\in\bR}\{xt-\bar f^*(t)\}
=\sup_{t>0}\{xt-\hat f(t)\}=\hat{\hat f}(x).
$$

(b)\enspace
To prove the assertion, we may assume that $B\in\bM_n^+$ is diagonal so that
$B=\diag(b_1,\dots,b_n)$ with $b_1\ge\dots\ge b_n$. Since $f(x)=\sup_{t>0}\{tx-\hat f(t)\}$
for $x\ge0$, we have
\begin{align*}
\Tr f(B)&=\sum_{i=1}^nf(b_i)
=\sup_{a_1,\dots,a_n>0}\sum_{i=1}^n\bigl\{a_ib_i-\hat f(a_i)\bigr\} \\
&=\sup_{A=\diag(a_1,\dots,a_n)\in\bP_n}\bigl\{\Tr AB-\Tr\hat f(A)\bigr\} \\
&\le\sup_{A\in\bP_n}\bigl\{\Tr AB-\Tr\hat f(A)\bigr\}.
\end{align*}
On the other hand, for every $A\in\bP_n$ with eigenvalues $a_1\ge\dots\ge a_n$, since
$\Tr AB\le\sum_{i=1}^na_ib_i$ by majorization (see, e.g., \cite[(III.19)]{Bh1},
\cite[Corollary 4.3.5]{Hi2}), we have
$$
\Tr AB-\Tr\hat f(A)\le\sum_{i=1}^n\bigl\{a_ib_i-\hat f(a_i)\bigr\}
\le\sum_{i=1}^nf(b_i)=\Tr f(B),
$$
and so
$$
\sup_{A\in\bP_n}\bigl\{\Tr AB-\Tr\hat f(A)\bigr\}\le\Tr f(B).
$$

(c)\enspace
Let $f$ be a non-constant and non-decreasing function on $(0,\infty)$ and $0<r<1$. Assume that
$\tilde f(x):=f(x^{1-r})$ is convex on $(0,\infty)$. Then it immediately follows that $f$ is
convex on $(0,\infty)$. Since $f(x^{1-r})/x^{1-r}=(\tilde f(x)/x)x^r\to+\infty$ as $x\to\infty$,
we have $f\in\cF_\convex^\nearrow(0,\infty)$. To show concavity of $\hat f(x^r)$, we can assume
that $f$ is $C^2$ (even $C^\infty$) on $(0,\infty)$. Indeed, let $\phi$ be a $C^\infty$
function on $\bR$ supported on $[-1,1]$ such that $\phi(x)\ge0$ and $\int_{-1}^1\phi(x)\,dx=1$.
For each $\eps>0$ define a function $f_\eps$ on $(0,\infty)$ by
\begin{equation}\label{F-A.1}
f_\eps(x):=\int_{-1}^1\phi(t)f(xe^{-\eps t})\,dt,\qquad x\in(0,\infty).
\end{equation}
Note that this product type regularization $f_\eps$ is $h_\eps(\log x)$, $x>0$, where $h_\eps$
is the usual (additive type) regularization $h_\eps(s):=\int_{-1}^1\phi(t)h(s-\eps t)\,dt$ of
$h(s):=f(e^s)$, $s\in\bR$ (see, e.g., \cite[pp.\,146--147]{Bh1}, \cite[Appendix A.2]{Hi2}).
Then, $f_\eps$ is $C^\infty$ on $(0,\infty)$ and $f_\eps\to f$ as $\eps\searrow0$ uniformly on
any bounded closed interval of $(0,\infty)$. It is clear that $f_\eps$ satisfies the same
assumption as $f$. Moreover, we see that $\hat f_\eps\to\hat f$ as $\eps\searrow0$ uniformly on
any bounded closed interval of $(0,\infty)$, whose proof is given in Lemma \ref{L-A.3} below
for completeness. So we may prove the conclusion for $f_\eps$ in place of $f$.

By taking the limit as $x\to0+$ of the equation
$$
{d\over dx}\,f(x^{1-r})=(1-r)x^{-r}f'(x^{1-r}),
$$
we see that $f'(0+)=0$. We can approximate $f$ by $g_\eps(x):=f(x)+\eps x^{1/(1-r)}$ for
$\eps>0$ so that $g_\eps$ satisfies the same assumption as $f$ and $\hat g_\eps(x)\to\hat f(x)$
as $\eps\searrow0$. Hence we furthermore assume that $f''(x)>0$ for all $x>0$ and so $f'(x)$ is
strictly increasing on $(0,\infty)$. Now, compute the second derivative of $f(x^{1-r})$ as
\begin{equation}\label{F-A.2}
{d^2\over dx^2}\,f(x^{1-r})=(1-r)x^{-r-1}
\bigl\{(1-r)x^{1-r}f''(x^{1-r})-rf'(x^{1-r})\bigr\},
\end{equation}
and therefore
\begin{equation}\label{F-A.3}
(1-r)x^{1-r}f''(x^{1-r})-rf'(x^{1-r})\ge0,\qquad x>0.
\end{equation}
For every $t>0$, since $f'(0+)=0$ and $\lim_{x\to\infty}f'(x)=\lim_{x\to\infty}f(x)/x=+\infty$,
there is a unique $x_0>0$ such that $f'(x_0)=t$ and thus $xt-f(x)$ on $x>0$ takes the maximum
at $x=x_0=(f')^{-1}(t)$. Hence
$$
\hat f(t)=t(f')^{-1}(t)-f\bigl((f')^{-1}(t)\bigr).
$$
We further compute
\begin{align*}
{d\over dt}\,\hat f(t^r)
&=rt^{r-1}(f')^{-1}(t^r)+t^r{rt^{r-1}\over f''\bigl((f')^{-1}(t^r)\bigr)}
-f'\bigl((f')^{-1}(t^r)\bigr){rt^{r-1}\over f''\bigl((f')^{-1}(t^r)\bigr)} \\
&=rt^{r-1}(f')^{-1}(t^r), \\
{d^2\over dt^2}\,\hat f(t^r)
&=r(r-1)t^{r-2}(f')^{-1}(t^r)+rt^{r-1}{rt^{r-1}\over f''\bigl((f')^{-1}(t^r)\bigr)} \\
&={rt^{r-2}\over f''\bigl((f')^{-1}(t^r)\bigr)}
\bigl\{(r-1)(f')^{-1}(t^r)f''\bigl((f')^{-1}(t^r)\bigr)+rt^r\bigr\}.
\end{align*}
Letting $x:=(f')^{-1}(t^r)$ so that $t^r=f'(x)$, we have
\begin{equation}\label{F-A.4}
{d^2\over dt^2}\,\hat f(t^r)
={rt^{r-2}\over f''(x)}\bigl\{(r-1)xf''(x)+rf'(x)\bigr\},
\end{equation}
which is $\le0$ thanks to \eqref{F-A.3} and $f''(x)>0$. Hence $\hat f(x^r)$ is concave on
$(0,\infty)$.

To prove the converse, assume that $f$ and hence $\hat f$ are in
$\cF_\convex^\nearrow(0,\infty)$. By interchanging $f$ with $\hat f$ and $r$ with $1-r$ it
suffices to prove that if $f(x^{1-r})$ is concave on $(0,\infty)$, then $\hat f(x^r)$ is convex
on $(0,\infty)$. By Lemma \ref{L-A.3} we can assume as in the first part of the proof that $f$
is $C^2$ on $(0,\infty)$. By approximating $f$ by $f(x)+\eps x^{1/(1-r)}$ as $\eps\searrow0$,
we can furthermore assume that $f''(x)>0$ for all $x>0$. From \eqref{F-A.2} we have
\begin{equation}\label{F-A.5}
(1-r)xf''(x)-rf'(x)\le0,\qquad x>0.
\end{equation}
Let
$$
\alpha:=f'(0+)=\lim_{x\to0+}f'(x)\in[0,+\infty).
$$
It is clear that $\hat f(t)=-f(0+)$ for all $t\in(0,\alpha]$ (if $\alpha>0$). So it remains to
prove that
\begin{equation}\label{F-A.6}
{d^2\over dt^2}\,\hat f(t^r)\ge0,\qquad t>\alpha^{1/r}.
\end{equation}
When $t>\alpha^{1/r}$, i.e., $t^r>\alpha$, we can define $x:=(f')^{-1}(t^r)$ and compute
\eqref{F-A.4} in the same way as above. Hence \eqref{F-A.6} follows from \eqref{F-A.5}.
\end{proof}

Concerning the assertion (c) above we need in Section 5 its ``only if\," part only while we
give it as ``if and only if\," for completeness.

\begin{lemma}\label{L-A.2}
\begin{itemize}
\item[\rm(a)] For each $f\in\cF_\concave^\nearrow(0,\infty)$ define
$$
\check f(t):=\inf_{x>0}\{xt-f(x)\},\qquad t\in(0,\infty).
$$
Then $\check f\in\cF_\concave^\nearrow(0,\infty)$ and $f\mapsto\check f$ is an
involutive bijection on $\cF_\concave^\nearrow(0,\infty)$, i.e., $\check{\check f}=f$
for all $f\in\cF_\concave^\nearrow(0,\infty)$.
\item[\rm(b)] For every $f\in\cF_\concave^\nearrow(0,\infty)$ and $B\in\bP_n$,
$$
\Tr f(B)=\inf_{A\in\bP_n}\bigl\{\Tr AB-\Tr\check f(A)\bigr\}.
$$
\item[\rm(c)] Let $f$ be a non-decreasing function on $(0,\infty)$ and $r>0$. If $f(x^{1+r})$
is concave on $(0,\infty)$, then $f\in\cF_\concave^\nearrow(0,\infty)$ and $\check f(x^{-r})$
is convex on $(0,\infty)$.
\item[\rm(d)] Let $f\in\cF_\concave^\nearrow(0,\infty)$ and $r>0$. If $f(x^{1+r})$ is convex
on $(0,\infty)$, then $\check f(x^{-r})$ is concave on $(0,\infty)$.
\end{itemize}
\end{lemma}

\begin{proof}
(a)\enspace
Let $f\in\cF_\concave^\nearrow(0,\infty)$ and $t\in(0,\infty)$. Since $xt-f(x)$ is convex in
$x\in(0,\infty)$ and $xt-f(x)=x(t-f(x)/x)\to+\infty$ as $x\to\infty$, it follows that
$\check f(t)$ is defined as a finite value. By definition, $\check f$ is concave and
non-decreasing. For any $x>0$ fixed, since $\check f(t)/t\le x-f(x)/t\to x$ as
$t\to\infty$, we have $\lim_{t\to\infty}\check f(t)/t=0$, so
$\check f\in\cF_\concave^\nearrow(0,\infty)$. To show that $f\mapsto\check f$ is an
involutive bijection, we extend $f$ to $\bar f$ on the whole $\bR$ by
$\bar f(0):=\lim_{x\to0+}f(x)$ (possibly $-\infty$) and $\bar f(x)=-\infty$ for $x<0$.
Then $-\bar f$ is a lower semicontinuous convex function on $\bR$, and the conjugate function
$(-\bar f)^*(t):=\sup_{x\in\bR}\{xt+\bar f(x)\}$ is given as
$$
(-\bar f)^*(t)=\begin{cases}-\check f(-t) & \text{if $t<0$}, \\
f(\infty)\ (:=\lim_{x\to\infty}f(x)) & \text{if $t=0$}, \\
+\infty & \text{if $t>0$}.
\end{cases}
$$
Due to the duality of conjugate functions, we have for $x>0$,
$$
-f(x)=-\bar f(x)=\sup_{t\in\bR}\{xt-(-\bar f)^*(t)\}=\sup_{t<0}\{xt+\check f(-t)\}
$$
by taking account of $(-\bar f)^*(0)=f(\infty)=-\lim_{t\to0+}\check f(t)$. Therefore,
$$
f(x)=\inf_{t<0}\{x(-t)-\check f(-t)\}=\inf_{t>0}\{xt-\check f(t)\}=\check{\check f}(x).
$$

(b)\enspace
The proof is similar to that of Lemma \ref{L-A.1}\,(b). We may use the majorization
$\Tr AB\ge\sum_{i=1}^na_ib_{n+1-i}$ for $A,B\in\bP_n$ with the respective eigenvalues
$a_1\ge\dots\ge a_n$ and $b_1\ge\dots\ge b_n$.

(c)\enspace
Let $f$ be a non-decreasing function on $(0,\infty)$ and $r>0$. Assume that
$\tilde f(x):=f(x^{1+r})$ is concave on $(0,\infty)$. Then it immediately follows that $f$ is
concave on $(0,\infty)$. Since $f(x^{1+r})/x^{1+r}=(\tilde f(x)/x)/x^r\to0$ as $x\to\infty$,
we have $f\in\cF_\concave^\nearrow(0,\infty)$. To show convexity of $\check f(x^{-r})$, the
regularization \eqref{F-A.1} and Lemma \ref{L-A.3} below can be employed so that we may assume
that $f$ is $C^2$ on $(0,\infty)$. By approximating $f$ by $f(x)+\eps x^{1/(1+r)}$ as
$\eps\searrow0$, we may assume that $\lim_{x\to0+}f'(x)=+\infty$ and $f''(x)<0$ for all $x>0$
and so $f'(x)$ is strictly decreasing on $(0,\infty)$. Since
\begin{equation}\label{F-A.7}
{d^2\over dx^2}\,f(x^{1+r})=(1+r)x^{r-1}\bigl\{(1+r)x^{1+r}f''(x^{1+r})+rf'(x^{1+r})\bigr\},
\end{equation}
we have
\begin{equation}\label{F-A.8}
(1+r)xf''(x)+rf'(x)\le0,\qquad x>0.
\end{equation}
For every $t>0$, since $\lim_{x\to0+}f'(x)=+\infty$ and
$\lim_{x\to\infty}f'(x)=\lim_{x\to\infty}f(x)/x=0$, there is a unique $x_0>0$ such that
$f'(x_0)=t$ and thus $xt-f(x)$ on $x>0$ takes the minimum at $x=x_0=(f')^{-1}(t)$. We hence
have $\check f(t)=t(f')^{-1}(t)-f\bigl((f')^{-1}(t)\bigr)$ and, as in the proof of Lemma
\ref{L-A.1}\,(c),
$$
{d^2\over dt^2}\,\check f(t^{-r})
={rt^{-r-2}\over f''\bigl((f')^{-1}(t^{-r})\bigr)}
\bigl\{(1+r)(f')^{-1}(t^{-r})f''\bigl((f')^{-1}(t^{-r})\bigr)+rt^{-r}\bigr\}.
$$
Letting $x:=(f')^{-1}(t^{-r})$ so that $t^{-r}=f'(x)$ we have
\begin{equation}\label{F-A.9}
{d^2\over dt^2}\,\check f(t^{-r})
={rt^{-r-2}\over f''(x)}\bigl\{(1+r)xf''(x)+rf'(x)\bigr\},
\end{equation}
which is $\ge0$ thanks to \eqref{F-A.8}. Hence
$\check f(x^{-r})$ is
convex on $(0,\infty)$.

(d)\enspace
Assume that $f\in\cF_\concave^\nearrow(0,\infty)$ and $f(x^{1+r})$ is convex on $(0,\infty)$.
We may assume that $f$ is $C^2$ as before. Approximating $f$ by $f(x)+\eps x^{1/(1+r)}$ as
$\eps\searrow0$ we may further assume that $f''(x)<0$ for all $x>0$. From \eqref{F-A.7} we have
$(1+r)xf''(x)+rf'(x)\ge0$ for all $x>0$. Let $\alpha:=\lim_{x\to0+}f'(x)\in(0,+\infty]$. If
$\alpha<+\infty$, then $f(0+):=\lim_{x\to0+}f(x)$ exists in $\bR$ and $\check f(t)=-f(0+)$ for
all $t\ge\alpha$. So it suffices to prove that
$$
{d^2\over dt^2}\,\check f(t^{-r})\le0,\qquad
t^{-r}<\alpha,\ \ \mbox{i.e.,}\ \ t>\alpha^{-1/r},
$$
which indeed holds since we have \eqref{F-A.9} with $x:=(f')^{-1}(t^{-r})$ when $t^{-r}<\alpha$.
\end{proof}

\begin{lemma}\label{L-A.3}
Let $f\in\cF_\convex^\nearrow(0,\infty)$ (resp., $f\in\cF_\concave^\nearrow(0,\infty)$) and
$f_\eps$ be defined by \eqref{F-A.1} for each $\eps>0$. Then
$f_\eps\in\cF_\convex^\nearrow(0,\infty)$ (resp., $f_\eps\in\cF_\concave^\nearrow(0,\infty)$)
and $\hat f_\eps\to\hat f$ as $\eps\searrow0$ uniformly on any bounded closed interval of
$(0,\infty)$.
\end{lemma}

\begin{proof}
Assume that $f\in\cF_\convex^\nearrow(0,\infty)$. By definition \eqref{F-A.1} it is obvious
that $f_\eps$ is non-decreasing and convex on $(0,\infty)$. It is also obvious that
$\lim_{x\to\infty}f_\eps(x)/x=+\infty$ follows from the same property of $f$. Hence
$f_\eps\in\cF_\convex^\nearrow(0,\infty)$ for any $\eps>0$. To prove the latter assertion, it
suffices to show that $f_\eps(t)\to f(t)$ for every $t>0$, for it is plain to see that a
pointwise convergent sequence of convex functions on $(0,\infty)$ is equicontinuous on any
bounded closed interval of $(0,\infty)$. Let $t>0$ be arbitrary. Choose a $\xi\ge0$ such that
$\hat f(t)=\xi t-f(\xi)$ (where $f(0)=f(0+)$). For every $\delta>0$ we have
$|f_\eps(\xi)-f(\xi)|<\delta$ and so $\hat f_\eps(t)\ge\xi t-f_\eps(\xi)\ge\hat f(t)-\delta$
for any sufficiently small $\eps>0$. Hence $\liminf_{\eps\searrow0}\hat f_\eps(t)\ge\hat f(t)$.
Now, suppose by contradiction that $\hat f_\eps(t)\not\to\hat f(t)$ as $\eps\searrow0$; then
there are a $\delta_0>0$ and a sequence $0<\eps_n\searrow0$ such that
$\hat f_{\eps_n}(t)\ge\hat f(t)+\delta_0$ for all $n$. Choose a sequence $x_n\ge0$ such that
$\hat f_{\eps_n}(t)=x_nt-f_{\eps_n}(x_n)$. By taking a subsequence we may assume that
$x_n\to x_0\in[0,\infty]$. If $x_0=0$, then
$$
\hat f(t)+\delta_0\le\hat f_{\eps_n}(t)=x_nt-f_{\eps_n}(x_n)
\longrightarrow-f(0)\le\hat f(t)\quad\mbox{as $n\to\infty$},
$$
a contradiction. If $x_0=\infty$, then we have a contradiction by taking the limit as
$n\to\infty$ of $\hat f_{\eps_n}(t)/x_n=t-f_{\eps_n}(x_n)/x_n$. Indeed, since
$\hat f_{\eps_n}(t)$ is lower bounded, the left-hand side tends to $0$ while
$f_{\eps_n}(x_n)/x_n\to+\infty$ so that the right-hand tends to $-\infty$. Therefore,
$x_0\in(0,\infty)$, so we can assume that $x_n$'s are in a bounded interval $[a,b]$ of
$(0,\infty)$. Since $f_{\eps_n}\to f$ uniformly on $[a,b]$, we have a contradiction again since
$$
\hat f(t)+\delta_0\le x_nt-f_{\eps_n}(x_n)\longrightarrow x_0t-f(x_0)\le\hat f(t).
$$
It thus follows that $\hat f_\eps(t)\to\hat f(t)$ as $\eps\searrow0$.

Next, assume that $f\in\cF_\concave^\nearrow(0,\infty)$. We have
$f_\eps\in\cF_\concave^\nearrow(0,\infty)$ for any $\eps>0$ similarly to the above case. For
every $t>0$ choose a $\xi\ge0$ such that $\check f(t):=\xi t-f(\xi)$ (where $f(0)=f(0+)$ if
$f(0+)>-\infty$). For every $\delta>0$ we have
$\check f_\eps(t)\le\xi t-f_\eps(\xi)\le\check f(t)+\delta$ for any sufficiently small $\eps>0$.
Hence $\limsup_{\eps\searrow0}\check f_\eps(t)\le\check f(t)$. Suppose that
$\check f_\eps(t)\not\to\check f(t)$ as $\eps\searrow0$; then
$\check f_{\eps_n}(t)\le\check f(t)-\delta_0$ for some $\delta_0>0$ and some sequence
$\eps_n\searrow0$. Then we have a contradiction as in the proof of the above case, whose details
are omitted here.
\end{proof}

\begin{example}\label{E-A.4}\rm
(1)\enspace
Let $0<r<1$. Besides $f(x)=x^s$ with $s\ge1/(1-r)$ the following are examples of non-decreasing
convex functions $f$ such that $f(x^{1-r})$ is convex on $(0,\infty)$:
\begin{itemize}
\item For any $s\ge1/(1-r)$ and $\alpha>0$, $f(x)=(x-\alpha)_+^s$ or $f(x)=(x^s-\alpha^s)_+$.
\item For $s_1,s_2\ge1/(1-r)$ and $\alpha>0$,
$$
f(x)=\begin{cases}x^{s_1} & \text{if $0<x\le\alpha$}, \\
\beta(x^{s_2}-\alpha^{s_2})+\alpha^{s_1} & \text{if $x\ge\alpha$},
\end{cases}
$$
where $\beta\ge(s_1/s_2)\alpha^{s_1-s_2}$.
\end{itemize}

(2)\enspace
Let $r>0$. Besides $f(x)=x^s$ with $0<s\le1/(1+r)$ and $f(x)=\log x$ the following are examples
of non-decreasing concave functions $f$ such that $f(x^{1+r})$ is concave on $(0,\infty)$:
\begin{itemize}
\item For any $0<s\le1/(1+r)$ and $\alpha>0$,
$$
f(x)=\begin{cases}x^s-\alpha x &
\text{if $0<x\le(s/\alpha)^{1/(1-s)}$}, \\
(1-s)(s/\alpha)^{s/(1-s)} & \text{if $x\ge(s/\alpha)^{1/(1-s)}$}.
\end{cases}
$$
\item For $0<s_1,s_2\le1/(1+r)$ and $\alpha>0$,
$$
f(x)=\begin{cases}x^{s_1} & \text{if $0<x\le\alpha$}, \\
\beta(x^{s_2}-\alpha^{s_2})+\alpha^{s_1} & \text{if $x\ge\alpha$},
\end{cases}
$$
where $0<\beta\le(s_1/s_2)\alpha^{s_1-s_2}$.
\end{itemize}
\end{example}

\end{document}